\documentclass[11pt, letterpaper]{amsart}

\usepackage{latexsym, amsmath, amstext, amssymb, amsfonts, amscd, bm, array, multirow, amsbsy, mathrsfs}
\usepackage{amsthm}
\usepackage{t1enc}
\usepackage[mathscr]{eucal}
\usepackage{indentfirst}
\usepackage{pb-diagram}
\usepackage{graphicx}
\usepackage{fancyhdr}
\usepackage{fancybox}
\usepackage{enumerate}
\usepackage{color}
\usepackage[all]{xy}
\usepackage{hyperref}
\usepackage{tikz}
\usepackage{xparse}
\usetikzlibrary{matrix}

\usepackage{url}
\usepackage[sort&compress,comma]{natbib}
\bibpunct{[}{]}{,}{n}{}{,}
\theoremstyle{plain}
\newtheorem{thm}{Theorem}[section]

\newtheorem{prop}[thm]{Proposition}

\newtheorem{lemma}[thm]{Lemma}

\theoremstyle{definition}
\newtheorem{defi}[thm]{Definition}

\theoremstyle{remark}
\newtheorem{remark}[thm]{Remark}
\newtheorem{ep}[thm]{Example}

\newcommand{\ZZ}{\ensuremath{\mathbb Z}}

\newcommand{\RR}{\ensuremath{\mathbb R}}
\newcommand{\g}{\ensuremath{\mathfrak{g}}}

\newcommand{\h}{\ensuremath{\mathfrak{h}}}

\newcommand{\li}{\ensuremath{L_{\infty}}}

\definecolor{forest}{rgb}{0,0.5,0} 

\newcommand{\dw}{\ensuremath{d_{tot}}}
\newcommand{\vs}{\varsigma}

\newcommand{\ham}[2]{\Omega^{#1}_{\mathrm{Ham}}\left(#2\right)}

\newcommand{\alphak}[1]{\alpha_{1} \otimes \cdots \otimes \alpha_{#1}}
\newcommand{\alphadk}[1]{\alpha_{1},\hdots,\alpha_{#1}}

\newcommand{\vk}[1]{v_{\alpha_{1}} \wedge \cdots \wedge  v_{\alpha_{#1}}}

\newcommand{\prim}{\varphi}
\newcommand{\tensor}{\otimes}

\newcommand{\maps}{\colon}

\begin{document}

\title{Products of multisymplectic manifolds\\ and homotopy moment maps}
  
  \author{C. S. Shahbazi}
\email{carlos.shabazi@cea.fr}
\address{Institut de Physique Théorique, CEA Saclay France.}

\author{M. Zambon}
\email{marco.zambon@wis.kuleuven.be}
\address{KU Leuven, Department of Mathematics, Celestijnenlaan 200B box 2400, BE-3001 Leuven, Belgium.}


\begin{abstract}

Multisymplectic geometry admits an operation that has no counterpart in symplectic geometry, namely, taking the product of two multisymplectic manifolds endowed with the \emph{wedge product} of the multisymplectic forms. We show that there is an $L_{\infty}$-embedding of the  $L_{\infty}$-algebra of observables of the individual factors into the observables of the product, and that homotopy moment maps for the individual factors induce a homotopy moment map for the product. As a by-product, we associate to every multisymplectic form  a curved $L_{\infty}$-algebra, whose curvature is the multisymplectic form itself.
  \end{abstract}

\maketitle

\setcounter{tocdepth}{1} 
\tableofcontents


\section*{Introduction}

    
Multisymplectic forms are higher analogues of symplectic forms. More precisely, we will refer to closed non-degenerate $(n+1)$-forms as $n$-plectic forms, so that for $n=1$ {we recover the definition of a symplectic form}. Although multisymplectic forms have been studied for a long time, in part due to the role they play in field theory, it was only around 2010 that the algebraic structure underlying them was unveiled: in \cite{BHR}\cite{RogersL} it was realized  that the ``observables'' on a multisymplectic manifold carry the structure of an $L_{\infty}$-algebra, which in the symplectic case reduces to the Poisson algebra of functions. Recall that an $L_{\infty}$-algebra is the notion that one obtains from a Lie algebra when one requires the Jacobi identity {to be satisfied only up to a higher coherent chain homotopy}. Given an $n$-plectic manifold $(M,\omega)$, we denote by $L_{\infty}(M,\omega)$ its associated $L_{\infty}$-algebra.

Given an action of a Lie group on a multisymplectic manifold  $(M,\omega)$, homotopy moment maps were introduced in \cite{FRZ} making use of $L_{\infty}(M,\omega)$ in an essential way. Homotopy moment maps enjoy nice properties:  cocycles in equivariant cohomology give rise to  homotopy moment maps, and  the latter are well-behaved w.r.t loop space constructions, as shown in \cite{FRZ}. In the setting of (higher) Hamiltonian systems, one can show  that homotopy moment maps induce conserved quantities \cite{Cons}. In the setting of (higher) prequantization,   homotopy moment maps can {be} lifted to higher prequantum bundles \cite{higherpreq}.
  
One feature of  multisymplectic geometry, {first explored in \cite{ShahbaziMaster}}, is that it admits a natural operation which has no counterpart in symplectic geometry, namely the wedge product: let $(M_a, \omega_a)$ be a $n_a$-plectic manifold, and similarly let $(M_b, \omega_b)$ be a $n_b$-plectic manifold. Then 

\begin{equation}
(\tilde{M},\tilde{\omega}):= (M_{a}\times M_{b}, \omega_{a}\wedge\omega_{b})
\end{equation}

\noindent
is also a multisymplectic manifold, since $\omega$ is a non-degenerate $(n_a+n_b+2)$-form. Notice that while this structure is natural and always well-defined, the structure on $\tilde{M}$  that is familiar from symplectic geometry -- namely the sum $\omega_{a}+\omega_{b}$ --  is of little use since it is not a form of well-defined degree except in the case $n_a=n_b$.\\

The main goal of this letter is to show that both the $L_{\infty}$-algebra of observables and  homotopy moment maps are well-behaved with respect to the above wedge product operation in multisymplectic geometry. Actually, all our results are proven in the more general setting of closed forms, in which the non-degeneracy assumption is dropped.

More precisely, assuming that a Lie group $G_{C}$, with Lie algebra $\mathfrak{g}_{C}$,  acts on $\left( M_{C}, \omega_{C}\right)$  with   homotopy moment map $f^{C}:\mathfrak{g}_{C}\to L_{\infty}\left(M_{C}, \omega_{C}\right)$, for $C=a\, ,b$:

\begin{enumerate}

\item We construct a homotopy moment map $$F: \mathfrak{g}_{a}\oplus\mathfrak{g}_{b}\to L_{\infty}(\tilde{M},\tilde{\omega})$$ for the product manifold $\left(\tilde{M},\tilde{\omega}\right)$, out of the homotopy moment maps $f^{C}$
for the individual factors.
 
\item We construct an $L_{\infty}$-embedding $$H : L_{\infty}(M_{a},\omega_{a})\oplus L_{\infty}(M_{b},\omega_{b}) \to L_{\infty}(\tilde{M},\tilde{\omega})$$ from the direct sum of the $L_{\infty}$-algebras of the factors, to the $L_{\infty}$-algebra of the product manifold.

\end{enumerate}

\noindent 
We will see that the two questions addressed above are closely related. Indeed, rather than approaching directly question (2), we  first construct $F$ as in question (1), and using its explicit formula we are able to make an educated guess for $H$ as in question (2) so that the following diagram of $L_{\infty}$-morphisms commutes: 

\begin{center}
\begin{equation}
\begin{tikzpicture}[baseline=(current bounding box.center)]
\label{diag:productmoment}
  \matrix (m) [matrix of math nodes,row sep=8em,column sep=9em,minimum width=2em]
  {
L_{\infty}\left( M_{a},\omega_{a}\right)\oplus L_{\infty}\left( M_{b},\omega_{b}\right) & L_{\infty}\left( \tilde{M}, \tilde{\omega}\right) \\
\mathfrak{g}_{a}\oplus\mathfrak{g}_{b}  \\};
  \path[-stealth]
    (m-2-1) edge node [left] {$f^{a}\oplus f^b$} 
    (m-1-1)
    (m-1-1) edge node [above] {$\mathrm{H}$} (m-1-2)
    (m-2-1) edge node [below] {$F$} (m-1-2);
\end{tikzpicture}
\end{equation}
\end{center}

We construct the homotopy moment map $F$ out of $f^a$ and $f^b$ in \S \ref{sec:Momentmapproduct} (see Thm. \ref{thm:morphismproduct}), making use of the machinery developed in \cite{FLRZ}\cite{WurzRyvkinMomaps}, and we compare our construction with the one given by \cite{FRZ} for homotopy moment maps arising from equivariant cocycles.  
Then in \S \ref{section:iterated} we specialize to the case of iterated powers of the same multisymplectic form, i.e. $(M,\omega^m)$, displaying explicit formulae for the case   $(M,\omega^2)$ and discussing Hyperk\"ahler manifolds as an example. 
In \S \ref{section:emb}  we construct the $\li$-embedding $H$ (by $\li$-embedding we mean an $\li$-morphism whose first component $H_1$ is injective). We do this in Thm. \ref{thm:embed}, using the formulae for $F$  as a guide. 

Finally in \S \ref{section:curved} we present an interesting by-product of this note, namely, the existence of a \emph{curved} $L_{\infty}$-algebra that is naturally associated to every multisymplectic manifold, and whose ``curvature'' is the multisymplectic form. Being a genuinely   {curved} $L_{\infty}$-algebra, it differs from the  $L_{\infty}$-algebra of observables $L_{\infty}(M,\omega)$ introduced in \cite{BHR}\cite{RogersL}.
{The underlying graded vector spaces are the same in degrees $\le 0$, but the one of the  {curved} $L_{\infty}$-algebra also has non-trivial components in degrees $1$ and $2$, while $L_{\infty}(M,\omega)$  is trivial in those degrees.} \\

  
\noindent{\bf Acknowledgements:} We thank Chris Rogers for comments on a preliminary version of Section 5, and Martin Callies for sharing with us a draft of \cite{Callies}.
C.S. was partially supported by the ERC Starting Grant 259133 -- ObservableString. M.Z. was partially supported by grants  
MTM2011-22612 and ICMAT Severo Ochoa  SEV-2011-0087 (Spain), Pesquisador Visitante Especial grant  88881.030367/2013-01 (CAPES/Brazil) and  IAP Dygest (Belgium).


\section{Background on homotopy moment maps}

In this section we  briefly review the geometry of closed differential forms and the notion of homotopy moment map, which will be used through the rest of this note, following \cite{RogersThesis,FRZ}.  We will call  $(M,\omega)$ a {\bf pre-$n$-plectic} manifold if $M$ is a  manifold and $\omega$ a closed $(n+1)$-form. 

 
\subsection{Closed forms on manifolds and $\li$-algebras}


\begin{defi} \label{Hamiltonian}
Let $(M,\omega)$ be a pre-$n$-plectic manifold. A $(n-1)$-form $\alpha$ is said to be {\bf Hamiltonian} if and only if there exists a vector field $v_\alpha \in \mathfrak{X}(M)$ such that

\begin{equation*}
\label{eq:hamcondition}
d \alpha= -\iota_{v_\alpha} \omega\, .
\end{equation*}

\noindent
We say then that $v_\alpha$ is a {\bf Hamiltonian vector field} for $\alpha$. The sets of Hamiltonian $(n-1)$-forms and Hamiltonian vector fields are respectively denoted by $\ham{n-1}{M}$ and $\mathfrak{X}_{\mathrm{Ham}}(M)$.
\end{defi}
A pre-$n$-plectic manifold $(M,\omega)$ is said to be {\bf{$n$-plectic}} if for every  $ u \in TM$, the following non-degeneracy condition is satisfied: $\iota_u \omega = 0 $ implies $ u=0$. In other words, $\omega$ is injective when seen as a bundle map $ TM\to  \wedge^n T^{\ast} M $. Notice that if $(M,\omega)$ is $n$-plectic, then for each Hamiltonian form $\alpha\in\ham{n-1}{M}$ there exists a unique Hamiltonian vector field $v_{\alpha}\in \mathfrak{X}_{\mathrm{Ham}}(M)$.
Further, a $1$-plectic manifold is the same thing as a symplectic manifold.
\begin{defi} 
\label{def:bracket}
Let $(M,\omega)$ be a pre-$n$-plectic manifold. We define the bilinear bracket $\left\{\cdot , \cdot\right\}_{2} : \Omega^{n-1}_{\mathrm{Ham}}(M)\times \Omega^{n-1}_{\mathrm{Ham}}(M)\to \Omega^{n-1}_{\mathrm{Ham}}(M)$ as follows
\begin{equation*}
\left\{\alpha, \beta\right\}_{2} = \iota_{v_{\beta}}\iota_{v_{\alpha}}\omega\, , \qquad \alpha , \beta \in \Omega^{n-1}_{\mathrm{Ham}}(M)\, ,
\end{equation*}
where $v_{\alpha}$ and $v_{\beta}$ are any Hamiltonian vector fields for $\alpha$ and $\beta$ respectively. 
\end{defi}
The bracket of two Hamiltonian forms is Hamiltonian, and it is well defined since it does not depends on the choice of Hamiltonian vector field among those which are associated with the given Hamiltonian forms. Although the bracket is skew-symmetric, it fails to satisfy the Jacobi identity (the failure is given by an exact form), and therefore it does not make the vector space of Hamiltonian forms into a Lie algebra. Of course one could consider the  induced graded Lie bracket on the quotient of $\Omega^{n-1}_{\mathrm{Ham}}(M)$ by the exact forms or by the closed forms\footnote{The latter quotient is isomorphic to $\mathfrak{X}_{\mathrm{Ham}}(M)$ as a graded Lie algebra.}, but doing so one loses a lot of information.

In \cite[Thm. 5.2]{RogersL}, Rogers associated to any $n$-plectic manifold an $L_{\infty}$-algebra, depending exclusively on $\omega$ and the de Rham differential $d$. This was generalized  to pre-$n$-plectic manifolds in \cite[Thm. 6.7]{HDirac}. Let us first recall the general definition of $L_{\infty}$-algebra.

\begin{defi}[\cite{Lada-Markl}] 
\label{def:Linfinity} 
An {\bf $L_{\infty}$-algebra} is a graded vector space $L$ equipped with a collection $\left\{ l_{k} \maps L^{\tensor k} \to L \,\, | \,\, 1 \leq k < \infty \right\}$  of graded skew-symmetric linear maps with  $\deg{l_{k}}=2-k$, such that the following identity holds for $m\ge 1$ and homogeneous elements $x_1,\dots,x_m\in L$:
\begin{eqnarray*} 
   \sum_{\substack{i+j = m+1, \\ \sigma \in Sh_{i,m-i}}}
  (-1)^{\sigma}\epsilon(\sigma)(-1)^{i(j-1)} l_{j}
   (l_{i}(x_{\sigma(1)}, \dots, x_{\sigma(i)}), x_{\sigma(i+1)}\, ,
   \ldots, x_{\sigma(m)})=0\, .
\end{eqnarray*}
Here $Sh_{i,m-i}$ denotes the $(i,m-i)$-unshuffles, 
 i.e.
permutations $\sigma$ of $\{1,\dots,m\}$ such that $\sigma(1)<\dots<\sigma(i)$ and  $\sigma(i+1)<\dots<\sigma(m)$, while
 $\epsilon(\sigma)$ is the Koszul\footnote{The Koszul sign depends on $x_1,\dots,x_m$ too. For instance, if $\sigma$ is the transposition of $x_1$ and $x_2$, then
the Koszul sign is $(-1)^{|x_1|\cdot|x_2|}$.}  sign.
\end{defi}

\noindent
The definition of $L_{\infty}$-algebra may seem  somehow arbitrary, however it admits a conceptual and elegant formulation in terms of a coalgebra equipped with a codifferential \cite{Lada-Markl,FRZ}, which we will not need in this note. We will be interested in a particular class of $L_{\infty}$-algebras:

\begin{defi}
A {\bf Lie $n$-algebra} is an $L_{\infty}$-algebra $\left( L, l_{k}\right)$ such that the graded vector space $L$ is concentrated in degrees $-n+1,\dots,-1,0$.
\end{defi}
\noindent For Lie $n$-algebras\footnote{{Lie $n$-algebras should not be confused with 
Filippov's notion of $n$-Lie algebra, in which the structure is given by a single map of arity $n$.}},  by degree counting we have $l_{k} = 0$ for $k>n+1$. For $n=1$ we recover the definition of an ordinary Lie algebra. The $L_{\infty}$-algebra constructed in references \cite{RogersL,HDirac} starting from a pre-$n$-plectic or $n$-plectic manifold is indeed a particular instance of Lie $n$-algebra. The construction is the following.

\begin{defi}
\label{def:multisymLinfinity}  
Let  $(M,\omega)$ be a pre-$n$-plectic manifold. There is a Lie $n$-algebra structure
$\mathbf{L_{\infty}(M,\omega)} = \left( L,\{l_{k} \}_{k \geq 1}\right)$  on the graded vector space $L$ whose non-trivial components are
\[
L_{i} =
\begin{cases}
\ham{n-1}{M} & \hbox{ for } i=0,\\
\Omega^{n-1+i}(M) & \hbox{ for }   1-n \leq i \leq -1.
\end{cases}
\]
The Lie $n$-algebra structure is given by 
  the sequence of maps $\{l_{k} \}_{k \geq 1}$ defined by
\[ 
l_{1}(\alpha)=
\begin{cases}
d\alpha & \text{if $\deg{\alpha} \leq -1$}, \\
0  & \text{if
  $\deg{\alpha}=0$}, 
  \end{cases}
\]
and for all $k\geq 2$ by
\[
l_{k}(\alphadk{k}) =
\begin{cases}
0 & \text{if $\deg{\alphak{k}} \leq -1$}, \\
\vs(k) \iota(\vk{k}) \omega  & \text{if
  $\deg{\alphak{k}}=0$}.
  \end{cases}
\]
Above, $v_{\alpha_{i}}$ is any Hamiltonian vector field associated to $\alpha_{i} \in \ham{n-1}{M}$, and we define  $\varsigma(k):=-(-1)^{k(k+1)/2}$ (so $\varsigma(k)=1,1,-1-1,1,\dots$ for $k=1,2,3,4,5,\dots$).
\end{defi}
\noindent Notice that $l_{2}\left(\cdot , \cdot\right) = \left\{\cdot , \cdot\right\}_{2}$, so the $L_{\infty}$-algebra constructed above extends the bilinear bracket of Def. \ref{def:bracket}. We will often write $\{\dots\}_k$ instead of $l_k\, , k\geq 1$. 

We introduce a further sequence of operations on $L$, which turns out to be very handy for the purposes of this note.

\begin{remark}
\label{prop:iotaomega} 
{The operations $[\dots]_k$ on $L$ we introduce now} are labelled by integers $k\ge 0$, unlike the operations introduced in Def. \ref{def:multisymLinfinity}. The multilinear maps $[\dots]_{k}$ are closely related to the multibrackets of $L_{\infty}(M,\omega)$: for $k\ge 1$,

\begin{equation*}
[\alphadk{k}]_k=\{\alphadk{k}\}_k-\delta_{k,1}d\alpha_1\,,
\end{equation*}  

\noindent
where $\delta$ denotes the Kronecker delta. In particular, for $k\ge 2$,   $[\dots]_k$ and $\{\dots\}_k$ agree, {while $[\alpha]_1$ vanishes if $\deg{\alpha} <0$ and equals $-d\alpha$ when $\deg{\alpha} =0$. We also have  $[1]_0=-\omega$.}
 In Prop. \ref{prop:curved} we will see that the $[\dots]_k$ extend to a curved $L_{\infty}$-algebra structure.
\end{remark}
 {Explicitly,  the operations $[\dots]_k$ are given as follows:}

\begin{defi}\label{def:square}  
Let $(M,\omega)$ be a pre-$n$-plectic manifold. Let $L$ denote the graded vector space underlying $L_{\infty}(M,\omega)$.

For all $k\ge 0$, we define the multilinear maps ${[\dots]_k}\colon L^{\otimes k}\to\Omega^{n+1-k}(M)$ as follows:
\[
[\alphadk{k}]_k =
\begin{cases}
0 & \text{if $\deg{\alphak{k}} \leq -1$}, \\
\vs(k) \iota(\vk{k}) \omega  & \text{if
  $\deg{\alphak{k}}=0$},
  \end{cases}
\]

 \end{defi}


\subsection{Homotopy moment maps and group actions}\label{momap}

 
Let $(M,\omega)$ be a pre-$n$-plectic manifold and let $G$ be a Lie group, with corresponding Lie algebra $\mathfrak{g}$, that acts on $(M,\omega)$ preserving $\omega $. The Lie group $G$ acts on $\Omega^{\bullet}(M)$ from the left via $g\cdot \omega\mapsto (\psi_{g^{-1}})^{\ast} \omega$, where $\psi_{g}$ is the diffeomorphism associated to $g$. The corresponding infinitesimal action is a Lie-algebra homomorphism from the Lie algebra $\mathfrak{g}$ to the vector fields $\mathfrak{X}(M)$ on $M$, namely:

\begin{equation*}
v_{-}:\mathfrak{g}\to\mathfrak{X}(M)\, ,\quad x\mapsto v_{x}\, ,
\end{equation*}

\noindent
where\footnote{The notation we chose for the vector field $v_x$ (associated to $x\in \g$ by the infinitesimal action) is similar to the one chosen for   Hamiltonian vector fields $v_{\alpha}$ of a Hamiltonian form $\alpha$ (Def. \ref{Hamiltonian}). We hope this does not give rise to confusion.}
\begin{equation*}
v_{x}|_{p} = \frac{d}{dt}exp(-tx)\cdot p|_{t=0}\, , \quad \forall p \in M\, .
\end{equation*}

\noindent
We present now the concept of homotopy moment map, introduced in  \cite{FRZ}, which generalizes the comoment map construction that appears in symplectic geometry.

\begin{defi}
\label{def:homotopymoment}
A {\bf homotopy moment map} for the action of $G$ on $(M,\omega)$ is an $\li$-morphism $f\colon \g \to L_{\infty}(M,\omega)$ such that for all $x\in \g$
\begin{equation}
\label{eq:mom}
d f_1(x)=-\iota_{v_x}\omega.
\end{equation}

\noindent
An action is said to be {\bf Hamiltonian} if it admits a homotopy moment map.
\end{defi}

\begin{remark}
a) From equation \eqref{eq:mom}, we see that a necessary (but not sufficient) condition for an action of $G$ to be Hamiltonian is that, infinitesimally, it acts through Hamiltonian vector fields. Notice that $f$ is not required to satisfy any equivariance  properties.

b) Def. \ref{def:homotopymoment} is a generalization of the notion of comoment map  for the action of a Lie group on a symplectic manifold. Indeed, for $n=1$ we recover the standard definition of a comoment map as a Lie-algebra homomorphism from the Lie algebra $\mathfrak{g}$ to the Poisson algebra of functions on the symplectic manifold. 
\end{remark}

\noindent
A homotopy moment map is   a particular instance of $L_{\infty}$-morphism, and the latter is  a fairly complicated object to handle in general. Luckily enough, we only need to consider $L_{\infty}$-morphisms having as source a Lie algebra, and as target a Lie $n$-algebra with the property that its  higher brackets are non-trivial only in degree zero (this is Property (P) in \cite[\S 3.2]{FRZ}). By  \cite[Prop. 3.8]{FRZ} (see also  the text at the beginning of Section 5 
there), $f\colon \g \to L_{\infty}(M,\omega)$  being a $\li$-morphism means that it consists of components  $f_k\colon \wedge^k \g\to \Omega^{n-k}(M)$ (for $k=1,\dots,n$) satisfying:

\begin{multline} \label{main_eq_1}
\sum_{1 \leq i < j \leq k}
(-1)^{i+j+1}f_{k-1}([x_{i},x_{j}],x_{1},\ldots,\widehat{x_{i}},\ldots,\widehat{x_{j}},\ldots,x_{k})\\
=df_{k}(x_{1},\ldots,x_{k}) + \vs(k)\iota(v_{x_1}\wedge \cdots \wedge v_{x_k})\omega
\end{multline}
 for $2 \leq k \leq n$, as well as  
\begin{multline} \label{main_eq_2}
\sum_{1 \leq i < j \leq n+1}
(-1)^{i+j+1}f_{n}([x_{i},x_{j}],x_{1},\ldots,\widehat{x_{i}},\ldots,\widehat{x_{j}},\ldots,x_{n+1})
=\vs(n+1)\iota(v_{x_1}\wedge \cdots \wedge v_{x_{n+1}})\omega.
\end{multline}

\noindent 
{Notice that the right-most   term   of eq. \eqref{main_eq_1} 
is just $l_k(f_1(x_1),\dots,f_1(x_k))$, and similarly for \eqref{main_eq_2}.}
 As mentioned above, comoment maps for symplectic manifolds are particular cases of homotopy moment maps. Further examples of homotopy moment maps can be found in  \cite{FRZ} and \cite{WurzRyvkinMomaps}.


\section{Homotopy moment maps for cartesian products $(M_{a}\times M_{b},\omega_{a}\wedge\omega_{b})$} 
\label{sec:Momentmapproduct}


Let $\left( M_{C}, \omega_{C}\right)\, , C=a\, ,b\, ,$ be   a pre-$n$-plectic manifold and let $G_{C}$ be a Lie group, with Lie algebra $\mathfrak{g}_{C}$, which acts on $\left( M_{C}, \omega_{C}\right)$ in a Hamiltonian way, with corresponding homotopy moment map $f^{C}:\mathfrak{g}_{C}\to L_{\infty}\left(M_{C}, \omega_{C}\right)$. 
Then $G\equiv G_{a}\times G_{b}$ acts on  the  pre-$(n_a+n_b+1)$-plectic manifold\footnote{We will slightly abuse the notation, denoting a differential form on $M_C$ and its pullback to $M_a\times M_b$, via  the canonical projection, by the same symbol.}

\begin{equation*}
\left( M\equiv M_a\times M_b\;,\;\omega\equiv\omega_a\wedge\omega_b\right)\, .
\end{equation*}

\noindent  The main theorem of this section is Theorem \ref{thm:morphismproduct}, where from the above data we explicitly construct a homotopy moment map $F: \mathfrak{g}_{a}\otimes\mathfrak{g}_{b}\to L_{\infty}(M,\omega)$.  

\subsection{The construction of $F$}

 We first recall a few facts from \cite[\S 2]{FLRZ} \cite{WurzRyvkinMomaps}. Let $(M,\omega)$ be a   pre-$n$-plectic manifold, and $G$ a Lie group acting on $M$ preserving $\omega$. The manifold $M$ and the Lie algebra $\g$ give rise to a  double complex

\begin{equation*}
K:=(\wedge^{\ge1} \g^*\otimes \Omega(M), d_\g,d)\, ,
\end{equation*}

\noindent
where $d_\g$ is the Chevallier-Eilenberg differential of $\g$ and $d$ is the de Rham differential of $M$. We consider the total complex with differential $$\dw:=d_\g\otimes 1+1\otimes d.$$ Hence,  on an element of $\wedge^k \g^*\otimes \Omega(M)$, $\dw$ acts as $d_\g + (-1)^kd$.

For any $G$-invariant $\sigma\in \Omega^N(M)$ define 

\begin{equation*}
\label{eq:omegak}
{\sigma}^k \colon \wedge^k\g \to \Omega^{N-k}(M),\;\; (x_1,\dots,x_k)\mapsto \iota{(v_1\wedge\dots\wedge v_k)}\sigma\, ,
\end{equation*}
and 
\begin{equation}\label{eq:tild}
\tilde{\sigma}:=\sum_{k=1}^{N}(-1)^{k-1}{\sigma}^k.
\end{equation}
 Since each ${\sigma}^k$ can be viewed as an element of $\wedge^k\g^* \otimes \Omega^{N-k}(M)$, it follows that $\sigma$ can be viewed as an element of $K$ of total degree $N$. It turns out that $\tilde{\omega}$ is $\dw$-closed, as a consequence of the fact that $\omega$ is a closed form. The link to homotopy moment maps is given by \cite[Prop. 2.5]{FLRZ}, which we reproduce for the reader's convenience:
 
\begin{prop} \label{prop:doubleprimitive}
Let $\prim=\prim_1+\dots+\prim_n$, with $\prim_k \in \wedge^k \g^*\otimes \Omega^{n-k}(M)$. Then: $\dw\prim=\widetilde{\omega}$ if{f} $$f_k:=\varsigma(k)\prim_k \colon \wedge^k\g \to \Omega^{n-k}(M),$$ for $k=1,\dots,n$, are the components of a homotopy moment map for the action of $G$ on $(M,\omega)$.
\end{prop}

\noindent
Now we apply the previous machinery to the manifolds $M_a,M_b,M_a\times M_b$ and the data given at the beginning of this section. For each of these three manifolds we obtain a double complex, which we will denote by $(K_a,\dw^a)$, $(K_b,\dw^b)$ and  $(K,\dw)$ respectively. 
 
\begin{lemma}
\label{lemma:dtotprim}
Let $\prim^C\in K^C$ be of degree $n_C$. If $\dw^C\prim^C=\widetilde{\omega_C}$ for $C=a,b$, then  $\dw \prim=\widetilde{\omega_a\wedge\omega_b}$ where
\begin{equation*}
\prim = \frac{1}{2}(-\prim^a\widetilde{\omega_b}+(-1)^{n_a}\widetilde{\omega_a}\prim^b)+(\prim^a {\omega_b}+(-1)^{n_a+1} {\omega_a}\prim^b)\in K\, .
\end{equation*}
\end{lemma}

\begin{proof}
First notice that 

\begin{equation}\label{eq:3terms}
\widetilde{\omega_a\wedge\omega_b}=-\widetilde{\omega_a}\widetilde{\omega_b}+\widetilde{\omega_a}{\omega_b}+ {\omega_a}\widetilde{\omega_b}\, .
\end{equation}

\noindent
This is a consequence of $\widehat{\omega_a}\widehat{\omega_b}=\widehat{\omega_a\wedge\omega_b}$ for $\widehat{\omega_C}:={\omega_C}-\widetilde{\omega_C}$.

Now we exhibit $\dw$-primitives for each of the three summands in eq. \eqref{eq:3terms}.
\begin{align*}
\dw(\prim^a\widetilde{\omega_b}+(-1)^{n_a+1}\widetilde{\omega_a}\prim^b)&=
\dw^a\prim^a\widetilde{\omega_b}+(-1)^{n_a}\prim^a\dw^b\widetilde{\omega_b}
+(-1)^{n_a+1}\dw^a\widetilde{\omega_a}\prim^b+ \widetilde{\omega_a}\dw^b\prim^b\\
&=2\widetilde{\omega_a}\widetilde{\omega_b} 
\end{align*}
where in the last equation we used our assumption and $\dw^C\widetilde{\omega_C}=0$, which holds by \cite[\S 2]{FLRZ}.

Further $$\dw(\prim^a {\omega_b})=\dw^a\prim^a {\omega_b}+(-1)^{n_a}
\prim^a \dw^b{\omega_b}=\widetilde{\omega_a}{\omega_b},$$
where in the last equation to compute $ \dw^b{\omega_b}=0$ we have to enlarge the double complex $K^b$ to include $\wedge^{0} (\g_b)^*\otimes \Omega(M_b)\cong\Omega(M_b)$.
 
Similarly, $$\dw( (-1)^{n_a+1}{\omega_a}\prim^b)={\omega_a} \widetilde{\omega_b}.$$
\end{proof}

\noindent
Applying Prop. \ref{prop:doubleprimitive}, the $\dw$-primitive of $\widetilde{\omega_a\wedge\omega_b}$ obtained in Lemma \ref{lemma:dtotprim} allows us to construct a homotopy moment map for the $\g$ action on $(M,\omega_a\wedge\omega_b)$:

\begin{thm}
\label{thm:morphismproduct} Let $G_{C}$ be a Lie group with Lie algebra $\mathfrak{g}_{C}$, where $C=a,b$. Let $( M_{C},\omega_{C})$ be a pre-$n_{C}$-plectic manifold equipped with a $G_{C}$ action admitting a homotopy moment map  $f^{C}:\mathfrak{g}_{C}\to L_{\infty}\left( M_{C},\omega_{C}\right)$. Then 
the action of $G_{a}\times G_{b}$ on $\left( M,\omega\right):=(M_a\times M_b, \omega_a\wedge \omega_b)$
admits a homotopy moment map with components determined by graded skew-symmetry and the formulae 
($k=1,\dots , n_{1}+n_{2}+1$)
\begin{center}
\fbox{\begin{Beqnarray*}
 F_{k}:\left(\mathfrak{g}_{a}\oplus \mathfrak{g}_{b}\right)^{\otimes k} &\to & L_{\infty}\left( M,\omega\right)\,\nonumber\\ 
\left(x^{1}_{a}, \dots , x^{m}_{a},x^{1}_{b}, \dots, x^{l}_{b}\right) &\mapsto & c^{a}_{m,l}\,f^{a}_{m}\left(x^{1}_{a}, \dots , x^{m}_{a}\right)\wedge \iota_{1, \dots ,l}\omega_{b} \\&+& c^{b}_{m,l}\,\iota_{1, \dots ,m}\omega_{a} \wedge f^{b}_{l}\left(x^{1}_{b}, \dots , x^{l}_{b}\right)\nonumber\, ,\end{Beqnarray*}}
\end{center}

\noindent
where $m,l\ge 0$ with $m+l=k$, $x^{i}_{a}\in \g_a$ and $x^{i}_{b}\in \g_b$. Here we define   $f^{a}_{0} = f^{b}_{0} = 0$ and
\begin{equation*}
\iota_{1, \dots ,i}\,\omega_{C} = \iota\left(v_{f^{C}_{1}\left(x^{1}_{C}\right)}\wedge \dots \wedge v_{f^{C}_{1}\left(x^{i}_{C}\right)}\right)\omega_{C}\,.
\end{equation*}
 The coefficients are defined as follows for all $m \geq 1,l\geq 1$:
\begin{eqnarray*}
c^{a}_{m,l} & = &\frac{1}{2}\vs (m+l) \vs (m)(-1)^{(n_{a}+1-m)l}\, , \\c^{b}_{m,l} & = & \frac{1}{2}\vs (m+l) \vs (l)(-1)^{(n_{a}+1-m)(l+1)}\, ,
\end{eqnarray*}
and
\begin{equation*}
c^{a}_{m,0} = 1\, ,\qquad c^{b}_{0,l} = (-1)^{(l+1)(n_{a}+1)}\,.
\end{equation*}
Recall that $\vs(k) = -(-1)^{\frac{k(k+1)}{2}}$. 
\end{thm}
 
\begin{remark} The formula for $F_k$ simplifies once written
using the operations $[\dots]$ introduced in Def. \ref{def:square}:
\begin{align*}
F_k(x^{1}_{a}, \dots , x^{m}_{a},x^{1}_{b}, \dots, x^{l}_{b})&=
\widehat{c^{a}_{m,l}}\,f^{a}_{m}\left(x^{1}_{a}, \dots , x^{m}_{a}\right)\wedge 
\left[f_1^b(x^1_b),\dots,f_1^b(x^l_b)\right] \\&+ \widehat{c^{b}_{m,l}}\,\left[f_1^a(x^1_a),\dots,f_1^a(x^m_a)\right] \wedge f^{b}_{l}\left(x^{1}_{b}, \dots , x^{l}_{b}\right),
\end{align*}
where  for all $m \geq 1,l\geq 1$:
\begin{eqnarray*}
\widehat{c^{a}_{m,l}} & = &-\frac{1}{2}(-1)^{(n_a+1)l}\, ,\\
\widehat{c^{b}_{m,l}} & = & -\frac{1}{2}(-1)^{(n_a+1)(l+1)+m} 
\, ,
\end{eqnarray*}
and
\begin{equation*}
\widehat{c^{a}_{m,0}} = -1\, ,\qquad \widehat{c^{b}_{0,l}} = -(-1)^{(l+1)(n_{a}+1)}.
\end{equation*}
This is a straightforward consequence of $\vs(m)\vs(l)\vs(m+l)=-(-1)^{ml}$ for all integers $m,l\ge 0$.
\end{remark}

\begin{proof}
Prop. \ref{prop:doubleprimitive}
and Lemma \ref{lemma:dtotprim} deliver a homotopy moment map 
$F\colon \mathfrak{g}_{a}\oplus\mathfrak{g}_{b}\to L_{\infty}(M_a\times M_b,\omega)$ whose components $F_{k}$, for  $k=1,\dots,n_{a}+n_{b}+1,$ 
 are given by
\begin{equation*}
F_{k} = \vs(k)\prim_{k}\, ,
\end{equation*}
where 
\begin{equation}
\label{eq:primproof}
\prim = \frac{1}{2}(-\prim^a\widetilde{\omega_b}+(-1)^{n_a}\widetilde{\omega_a}\prim^b)+(\prim^a {\omega_b}+(-1)^{n_a+1} {\omega_a}\prim^b)\,.
\end{equation}
 Let us point out that 
\begin{equation*}
\prim_{k}\in \Lambda^{k}\left(\mathfrak{g}^{\ast}_{a}\oplus\mathfrak{g}^{\ast}_{b}\right)\otimes\Omega^{(n_{a}+n_{b}+1-k)}\left( M_{a}\times M_{b}\right)\, .
\end{equation*}
In order to prove the theorem we just have to write  $F_{k}$   using equation \eqref{eq:primproof} and  $f_k^a=\varsigma(k)\prim^a_k$, $f_k^b=\varsigma(k)\prim^b_k$. We do so {evaluating} the components of $F$ on elements of $\g_a$ and of $\g_b$.

We have
\begin{align*}
F_{m}(x_{a}^{1},\hdots,x_{a}^{m})=\vs(m) \prim_{m}(x_{a}^{1},\hdots,x_{a}^{m}) &=  \vs(m)\prim^a_m(x_{a}^{1},\hdots,x_{a}^{m})\wedge {\omega_b}\\
&=f^{a}_{m}(x^{a}_{1},\hdots,x^{a}_{m})\,\wedge \omega_{b}\, ,
\end{align*}

\noindent
using that $\prim^{a}_{m} = \vs(m) f^{a}_{m}$ in the last equality. In the second equality we used eq. \eqref{eq:primproof} (notice that on the r.h.s. of eq. \eqref{eq:primproof}, only the summand $\prim^a\omega_b$  gives a contribution). We conclude that

\begin{equation*}
c^{a}_{m,0} = 1\, , \qquad m\geq 1\, .
\end{equation*}

\noindent
Let us take now
\begin{align*}
F_{l}(x_{b}^{1},\hdots,x_{b}^{l}) =\vs(l) \prim_{l}(x_{b}^{1},\hdots,x_{b}^{l}) &= \vs(l) (-1)^{n_a+1} ({\omega_a}\prim^b_l)(x_{b}^{1},\hdots,x_{b}^{l})\\
 &=  (-1)^{n_a+1} ({\omega_a}f^b_l)(x_{b}^{1},\hdots,x_{b}^{l})\\
&=(-1)^{(n_a+1)(l+1)} {\omega_a}\wedge f^{b}_{l}(x_{b}^{1},\hdots,x_{b}^{l}).
\end{align*}

\noindent
The last equality holds since\footnote{We are slightly abusing the notation by denoting the product of two elements in the double-complexes $K_{C}$ or $K$ and the wedge product of forms simply by juxtaposition.}, if we pick a basis $\{\xi^b_i\}$ of $\g_b^*$ and write
$f^{b}_{l}$ as a sum of terms of the form $\xi^b_{i_1}\wedge\cdots\wedge\xi^b_{i_l}\otimes \beta \in \Lambda^{l}\left(\mathfrak{g}^{\ast}_{b}\right)\otimes\Omega^{(n_{b}-l)}\left( M_{b}\right)$, then
\begin{equation*}
(1\otimes{\omega_a}) (\xi^b_{i_1}\wedge\cdots\wedge\xi^b_{i_l}\otimes \beta) = (-1)^{(n_{a}+1)l}
\xi^b_{i_1}\wedge\cdots\wedge\xi^b_{i_l}\otimes ({\omega_a} \wedge \beta).
\end{equation*}

\noindent
We obtain 

\begin{equation*}
c^{b}_{0,l} = (-1)^{(n_{a}+1)(l+1)}\, , \qquad l\geq 1\, .
\end{equation*}

For $m,l\ge 1$ consider
\begin{align*}
&F_{m+l}(x_{a}^{1},\hdots,x_{a}^{m},x_{b}^{1},\hdots,x_{b}^{l})\\
 =& 
\vs(m+l)\prim_{m+l}(x_{a}^{1},\hdots,x_{a}^{m},x_{b}^{1},\hdots,x_{b}^{l})\\
=&\vs(m+l)\frac{1}{2}\left(-\prim^a_m(\widetilde{\omega_b})_l+(-1)^{n_a}(\widetilde{\omega_a})_m\prim^b_l\right)(x_{a}^{1},\hdots,x_{a}^{m},x_{b}^{1},\hdots,x_{b}^{l})\\
=&\vs(m+l)\frac{1}{2}\left(
-\vs(m)(-1)^{l-1}f^a_m{\omega_b}^l
+(-1)^{n_a}(-1)^{m-1}\vs(l){\omega_a}^m f^b_l\right)(x_{a}^{1},\hdots,x_{a}^{m},x_{b}^{1},\hdots,x_{b}^{l}),
\end{align*}
where in the last equality we used eq. \eqref{eq:tild}.
We have
\begin{equation*}
 (f^a_m {\omega_b}^{l})(x_{a}^{1},\hdots,x_{a}^{m},x_{b}^{1},\hdots,x_{b}^{l})=
 (-1)^{(n_{a}-m)l}f^a_{m}(x_{a}^{1},\hdots,x_{a}^{m}) \wedge
 {\omega_b}^{l}(x_{b}^{1},\hdots,x_{b}^{l})\, ,
\end{equation*}
using $f^a_m\in 
\Lambda\mathfrak{g}^{\ast}_{a} \otimes\Omega^{n_{a}-m}\left( M_{a}\right)$ and 
${\omega_b}^{l}\in 
\Lambda^{l} \mathfrak{g}^{\ast}_{b} \otimes\Omega \left( M_{b}\right)$.
Therefore

\begin{equation*}
c^{a}_{m,l} =\frac{1}{2}\vs(l+m)\vs(m) (-1)^{(n_{a}+1-m)l}\, .
\end{equation*}

\noindent
Similarly,
\begin{equation*}
({\omega_a}^{m}f^b_{l})(x_{a}^{1},\hdots,x_{a}^{m},x_{b}^{1},\hdots,x_{b}^{l})=
 (-1)^{(n_{a}+1-m)l}{\omega_a}^{m}(x_{a}^{1},\hdots,x_{a}^{m}) \wedge
 f^b_{l}(x_{b}^{1},\hdots,x_{b}^{l})\, ,
\end{equation*}
using ${\omega_a}^{m}\in 
\Lambda\mathfrak{g}^{\ast}_{a} \otimes\Omega^{n_{a}+1-m}\left( M_{a}\right)$ and 
$f^b_{l}\in 
\Lambda^{l} \mathfrak{g}^{\ast}_{b} \otimes\Omega \left( M_{b}\right)$.
Hence

\begin{equation*}
c^{b}_{m,l} =\frac{1}{2}\vs(l+m)\vs(l) (-1)^{(n_{a}+1-m)(l+1)}\, .
\end{equation*}
\end{proof}

\begin{ep}\label{ex:presym}
We spell out the homotopy  moment map   constructed in Thm. \ref{thm:morphismproduct} in the case that $M_a$ and $M_b$ are pre-symplectic manifolds, i.e. $n_a=n_b=1$. In that case $f^a\colon \g_a\to C^{\infty}(M_a)$ is an ordinary comoment map, just like $f^b$, and
$\left(M,\omega\right)$ is a pre-$3$-plectic manifold. One obtains 
\begin{align*}
F_1(x_a\oplus x_b)&=f^a(x_a)\cdot\omega_b+\omega_a \cdot f^b(x_b)\\
F_2(x_a\oplus x_b, y_a\oplus y_b)&=\frac{1}{2}\left(-f^a(x_a)\cdot\iota_{v_{f^{b}(y_b)}}\omega_b+ \iota_{v_{f^{a}(x_a)}}\omega_a\cdot f^b(y_b)\right)-
\left(x\leftrightarrow y \right)
\\
F_3(x_a\oplus x_b, y_a\oplus y_b, z_a\oplus z_b)&=-\frac{1}{2}\left(f^a(x_a)\cdot\iota_{v_{f^{b}(y_b)}\wedge v_{f^{b}(z_b)}}\omega_b+  \iota_{v_{f^{a}(x_a)}\wedge v_{f^{a}(y_a)}}\omega_a\cdot f^b(z_b)\right)+ c.p.
\end{align*}
where $x_C,y_C,z_C\in \g_C$ for $C=a,b$ and ``c.p.'' denotes cyclic permutations of $x,y,z$.
\end{ep}

\subsection{Non-associativity of the construction} 
The construction of  homotopy  moment maps for product manifolds given in
Thm. \ref{thm:morphismproduct} is not associative.
More precisely: for $C=a,b,c$ let $G_{C}$ be a Lie group with Lie algebra $\mathfrak{g}_{C}$, acting on a pre-$n_{C}$-plectic manifold  $( M_{C},\omega_{C})$   with homotopy  moment map $f^{C}:\mathfrak{g}_{C}\to L_{\infty}\left( M_{C},\omega_{C}\right)$. Denote by $f^a* f^b$ the
homotopy moment map for the action of $G_{a}\times G_{b}$ on $\left( M_a\times M_b,\omega_a\wedge \omega_b\right)$ constructed in Thm. \ref{thm:morphismproduct}.
Then \begin{equation}\label{fabc}
(f^a* f^b)* f^c\neq f^a * (f^b* f^c),
\end{equation}
as one can see from a straightforward computation using the fact that 
$c^{a}_{m,l}=\pm \frac{1}{2}$ for $m \geq 1,l\geq 1$.

Indeed, the construction of the  $\dw$-primitives  
 done in Lemma \ref{lemma:dtotprim} is also not associative: denote by   $\prim^C$ the elements of $K^C$ corresponding to the homotopy moment maps $f^C$ (via Prop. \ref{prop:doubleprimitive}).
If we  denote by $\prim^a* \prim^b$ the $\dw$-primitive
of  $\widetilde{\omega_a\wedge\omega_b}$ constructed in Lemma \ref{lemma:dtotprim},
then $(\prim^a* \prim^b)* \prim^c$ and $\prim^a* (\prim^b* \prim^c)$ are different\footnote{One could hope that redefining $\prim^a* \prim^b$ by adding a real multiple of $\dw(\prim^a \prim^b)$ to it might remove this issue, but this is not the case.} primitives 
for $\widetilde{\left((\omega_a\wedge\omega_b)\wedge\omega_c\right)}=
\widetilde{\left(\omega_a\wedge(\omega_b\wedge\omega_c)\right)}$.
The difference between these two primitives is
$$\frac{1}{4}(
-\prim^a\widetilde{\omega_b}\widetilde{\omega_c}
+\widetilde{\omega_a}\widetilde{\omega_b}\prim^c)=\dw\left(-\frac{1}{4}
\prim^a\widetilde{\omega_b}\prim^c\right).$$

\noindent
Hence the two homotopy  moment maps appearing in eq. \eqref{fabc} are \emph{inner equivalent} in the sense of \cite[Remark 7.10]{FLRZ}. This notion of  inner equivalence is the one that arises naturally considering the complex
$\wedge^{\ge 1} (\g_a\times \g_b\times \g_c)^*\otimes \Omega(M_a\times M_b\times M_c)$, and can be characterized as equivalence of $L_{\infty}$-morphisms (see \cite[Prop. A2]{FLRZ}).\\

Under quite restrictive conditions, there is another way to construct homotopy moment maps for product manifolds, which does have the property of being  associative in the sense above.

\begin{remark}
Given an action of $G_a$ on the pre-$n_a$-plectic manifold $(M_a,\omega_a)$, the theorem 
\cite[Thm. 6.8]{FRZ} provides a map
$$\Phi_{M_a}\colon \{\text{Closed extensions of $\omega_a$ in $C_{G_a}(M_a)$}\}\to \{\text{Homotopy moment maps for  $(M_a,\omega_a)$}\},$$
where $C_{G_a}(M_a)=(S \g_a^*\otimes \Omega(M_a))^{G_a}$ is the
Cartan model for the equivariant cohomology of the $G_a$ action on $M_a$ (it is a differential graded algebra).
This map is not surjective   in general \cite[\S 7.5]{FRZ}. It is also not injective in general: by the formulae in \cite[Thm. 6.8]{FRZ} it is clear that, if $\g_a$ is a  abelian Lie algebra, then  the component lying in $(S^2 \g_a^*\otimes \Omega^{n_a-3}(M_a))^{G_a}$ of a closed extension $\psi^a$   can not be recovered from the homotopy moment map $\Phi_{M_a}(\psi^a)$.

However, in the cases in which $\Phi_{M_a}$ and $\Phi_{M_b}$ are injective\footnote{The same prescription does not seem to work without the injectivity assumption, for in that case it seems to depend on the choice of $\psi^a$ and $\psi^b$. 
In view of the formulae in \cite[Thm. 6.8]{FRZ}, the technical reason behind this is the following: if $P_2^a\in S^2\g_a^*$ is a quadratic polynomial on the Lie algebra $\g_a$, then  the total skew-symmetrization of $P_2^a([\cdot,\cdot],[\cdot,\cdot])\colon \g_a^{\otimes 4}\to \RR$ does not seem to be determined by 
the total skew-symmetrization of $P_2^a(\cdot,[\cdot,\cdot])\colon \g_a^{\otimes 3}\to \RR$.},
one can carry out the following construction: 
if  homotopy  moment maps $f^C$ for $(M_C,\omega_C)$ arising from  closed extensions   in the Cartan model ($C=a,b$) are given, then 
\begin{equation}\label{eq:prescr}
\Phi_{M_a\times M_b}(\psi^a\cdot\psi^b)
\end{equation}
is a homotopy  moment map for $(M_a\times M_b,\omega_a\wedge \omega_b)$,   where $\psi^C$ is determined by $\Phi_{M_C}(\psi^C)=f^C$, and the dot denotes the product in the Cartan model $C_{G_a\times G_b}(M_a\times M_b)$.
 This prescription has the property of being  associative, in the sense above, for the simple reason that the algebra structure in the Cartan model is associative.
 
In the special case of pre-symplectic manifolds $(M_a,\omega_a)$ and $(M_b,\omega_b)$, the injectivity assumption is satisfied. The above prescription \eqref{eq:prescr} delivers a homotopy moment map $H$ for  $(M_a\times M_b,\omega_a\wedge \omega_b)$, which as expected 
is different from the one $F$ obtained in Ex. \ref{ex:presym}: we have $H_1=F_1$, 
$H_2=F_2$, but
\begin{align*}
H_3(x_a\oplus x_b, y_a\oplus y_b, z_a\oplus z_b)=&\frac{2}{3}F_3(x_a\oplus x_b, y_a\oplus y_b, z_a\oplus z_b)\\
&-\frac{1}{6}\left(f^a(x_a)f^b([y_b,z_b])+f^a([y_a,z_a])f^b(x_b)+c.p.\right)
\end{align*} 
where $x_C,y_C,z_C\in \g_C$ for $C=a,b$.
\end{remark}


\section{Application: homotopy moment maps for iterated powers $(M,\omega^m)$ }\label{section:iterated}


In Section \ref{sec:Momentmapproduct} we have shown how to build a homotopy moment map for the product manifold of two pre-multisymplectic manifolds, assuming that a homotopy moment map for the individual manifolds exist. Here we apply this construction to some specific examples of geometrical interest: powers of  closed forms and Hyperk\"ahler manifolds.

\subsection{Restrictions}  

 Let $G$ be a Lie group with Lie algebra $\g$, acting on a pre-$n$-plectic manifold $(M,\omega)$ with homotopy moment map $f\colon \g\to L_{\infty}(M,\omega)$.  One obtains new actions, either restricting to   a Lie subgroup of $G$ or to an invariant  submanifold of $(M,\omega)$. We display homotopy moment maps for both cases.  

\begin{lemma}\label{lem:restrict}
Let $H\subset G$ be a Lie subgroup, and denote by $j\colon \h\hookrightarrow \g$ the inclusion of its Lie algebra.
The restricted action of $H$ on $(M,\omega)$ has homotopy moment map $f \circ j\colon \h\to L_{\infty}(M,\omega)$.
\end{lemma}
\begin{proof} The Lie algebra morphism $j$ is in particular an $L_{\infty}$-morphism, so
$f \circ j$ also is. Since eq. \eqref{eq:mom} holds for all $x\in \g$, in particular it holds for all $x\in \h$.
\end{proof}

\begin{lemma}
\label{lemma:NenMGinvariant} 
Let $N\overset{i}{\hookrightarrow} M$ a $G$-invariant submanifold of $M$. Then
the action $G\circlearrowleft \left(N,i^{\ast}\omega\right)$ is Hamiltonian with homotopy moment map $i^{\ast} \circ f: \g\to L_{\infty}\left(N,i^{\ast}\omega\right)$.
\end{lemma}

\begin{proof}
According to Def. \ref{def:homotopymoment}, we have to show that
\begin{equation*}
f^{N} := i^{\ast}\circ f :\mathfrak{g}\to L_{\infty}\left(N,i^{\ast}\omega\right)
\end{equation*}
is an $L_{\infty}$-morphism such that 
 \begin{equation}
\label{eq:conditionhomotopyN}
-\iota_{(v_{x})^N} i^{\ast}\omega = d f^{N}_{1} (x)\, ,\qquad\forall\,\, x\in \mathfrak{g},
\end{equation}
 where $(v_{x})^N$, which is a generator of the action on $N$, denotes the restriction of the vector field $v_x$ to $N$.

Eq. \eqref{eq:conditionhomotopyN} follows simply by applying the pullback $i^*$ to Eq. \eqref{eq:mom}. To show that $f^N$ is an $L_{\infty}$-morphism, let us introduce the following $L_{\infty}$-subalgebra of  $L_{\infty}\left(M,\omega\right)$:

\begin{equation*}
L^{N}\left(M,\omega\right) = C^{\infty}\left(M\right)\oplus\Omega^{1}\left(M\right)\oplus\cdots\oplus
\widetilde{\Omega}^{n-1}_{\mathrm{Ham}}\left(M\right)\, ,
\end{equation*}

\noindent
where 

$$
\widetilde{\Omega}^{n-1}_{\mathrm{Ham}}\left(M\right) = \left\{\alpha\in\Omega^{n-1}_{\mathrm{Ham}}\left(M\right): \text{$\exists$  a Hamiltonian vector field of $\alpha$ tangent to $N$}\right\}.
$$

\noindent
Since $L_{\infty}\left(M,\omega\right)$ and $L^{N}\left(M,\omega\right)$ are equal in every component except for the degree zero component, in order to see that $L^{N}\left(M,\omega\right)$ is really a $L_{\infty}$-subalgebra of $L_{\infty}\left(M,\omega\right)$, we only have to check that the binary bracket $l_{2}$ of $L_{\infty}\left(M,\omega\right)$ restricts to $\widetilde{\Omega}^{n-1}_{\mathrm{Ham}}\left(M\right)$. This is indeed the case since given any two Hamiltonian forms $\alpha$ and $\beta$ and respective Hamiltonian vector fields $v_{\alpha},v_{\beta}$, a Hamiltonian vector field for 
$l_2(\alpha,\beta)$ is given by
the Lie bracket
$\left[ v_{\alpha},v_{\beta}\right]$, which of course is   tangent to $N$
whenever both $v_{\alpha}$ and $v_{\beta}$ are.

Notice   that the homotopy moment map 
$f\maps \mathfrak{g}\to  L_{\infty}\left(M,\omega\right)$
 takes values in $L_{\infty}^{N}\left(M,\omega\right)$, that is,
\begin{equation}
\label{eq:fginLN}
f_{k}(x)\in L^{N}\left(M,\omega\right)\, ,\qquad \forall\,\, x\in\mathfrak{g}^{\otimes k}\qquad k\geq 1\, .
\end{equation}
To prove this, since $L_{\infty}\left(M,\omega\right)$ and $L^{N}\left(M,\omega\right)$ are equal in every component but the zero one, we have to check equation \eqref{eq:fginLN} only in the $k=1$ case, that is, we have to prove that 
\begin{equation*}
f_{1}(x)\in \widetilde{\Omega}^{n-1}_{\mathrm{Ham}}\left(M\right)\, ,\qquad \forall\,\, x\in\mathfrak{g}\, .
\end{equation*}
It holds since a Hamiltonian vector field of $f_{1}(x)$ is the generator of the action  $v_x$, which is tangent to $N$ by assumption.

Next, notice that the pullback of forms  
$$i^*\colon L_{\infty}^{N}\left(M,\omega\right)  \to L_{\infty}\left(N,i^*\omega\right)$$
is\footnote{However the map $L_{\infty} \left(M,\omega\right)  \to L_{\infty}\left(N,i^*\omega\right)$ given by
pullback of forms is not an  $L_{\infty}$-morphism. This is the reason we need to introduce   $L_{\infty}^{N}\left(M,\omega\right)$.} a (strict) $L_{\infty}$-morphism, as a consequence of the facts that $i^*$ commutes with the de Rham differential and 
due
to the definition of $\widetilde{\Omega}^{n-1}_{\mathrm{Ham}}\left(M\right)$.
 We conclude that
$i^{\ast} \circ f: \g\to L_{\infty}\left(N,i^{\ast}\omega\right)$ is a homotopy moment map.
\end{proof}

\subsection{Actions on  $( {M}, {\omega}\wedge {\omega})$}

 Let us consider two pre-multisymplectic manifolds $(M_{C},\omega_{C})\, ,\,\, C=a, b$. We assume that there is a Hamiltonian action of a Lie group $G_{C}\circlearrowleft M_{C}$ with corresponding homotopy moment map $f^{C} : \mathfrak{g}_{C}\to L_{\infty}\left(M_C,\omega_C\right)$.
 By Thm. \ref{thm:morphismproduct} we know that there is also a Hamiltonian action 

\begin{equation}
\label{eq:productaction}
{G_a}\times {G_b}\circlearrowleft \left(M_{a}\times M_{b},\, \omega_{a}\wedge\, \omega_{b}\right)\, ,
\end{equation}

\noindent
with homotopy moment map $F$ given 
by Thm. \ref{thm:morphismproduct}. 

Assume now that  $G_{a} = G_{b} =:{G}$, whose Lie algebra we denote by  ${\mathfrak{g}}$.
One can restrict the action \eqref{eq:productaction} to the diagonal $\Delta {G} = \{ (g, g)\, : \,\, g\in {G}\}$ of ${G}\times {G}$:

\begin{equation}\label{deltaG}
 \Delta {G} \circlearrowleft (M_{a}\times M_{b},\, \omega_{a}\wedge\, \omega_{b} )\, .
\end{equation}
By Lemma \ref{lem:restrict}, a homotopy moment map for this action is
\begin{equation*}
F\circ j :\Delta {\mathfrak{g}}\to L_{\infty}\left(M_{a}\times M_{b},\, \omega_{a}\wedge\, \omega_{b}\right)\, ,
\end{equation*}
where 
\begin{equation}\label{eq:m}
j\colon \Delta {\mathfrak{g}}=\{(x,x):x\in {\g}\} \to {\mathfrak{g}}\oplus{\mathfrak{g}}
\end{equation}
 is the inclusion.
By the isomorphism ${G}\simeq\Delta {G}, g\mapsto (g,g)$ we can view eq. \eqref{deltaG} as an action of the Lie group ${G}$,
and $j$ as a map $ {\mathfrak{g}}\simeq \Delta {\mathfrak{g}}  \to {\mathfrak{g}}\oplus{\mathfrak{g}}$.

%
%
%
%
%
%
%
 
Now we specialize even further, taking $\,M_{a} =M_{b} =: {M}\,$, $\omega_{a} = \omega_{b} =: {\omega}$ and $f^a=f^b$.  

%
%
%

\noindent
The diagonal $\Delta {M} $ of ${M}\times {M}$ is invariant under the action of $\Delta {G}$. Therefore, using the   inclusion

\begin{eqnarray*}
i : \Delta {M} \hookrightarrow {M}\times {M}\, 
\end{eqnarray*}
and the identification $ {M}\simeq \Delta {M}$
we obtain by restriction an action of ${G}$ on ${M}$:

\begin{equation*}
{G}\simeq\Delta {G}\circlearrowleft \left(\Delta {M},i^{\ast}\left({\omega}\wedge{\omega}\right)\right)\simeq \left({M},{\omega}\wedge{\omega}\right).
\end{equation*}
Of course, this is interesting only when $\omega$ has even degree, for otherwise ${\omega}\wedge{\omega}=0$.
Lemma \ref{lemma:NenMGinvariant} states that this action is Hamiltonian with homotopy moment map given by  
\begin{equation*}
i^{\ast}F \circ j : {\mathfrak{g}}\to L_{\infty}\left({M},{\omega}\wedge{\omega}\right)\, ,
\end{equation*}

\noindent
where $F$ is as in theorem \ref{thm:morphismproduct}.
 
\begin{remark}
If an action ${G}\circlearrowleft ({M},{\omega})$ is Hamiltonian, 
then the action ${G}\circlearrowleft ({M},{\omega}^{m})\, ,\, m\in\mathbb{N}\, ,$ is also Hamiltonian. This follows from a
 slight variation of the above reasoning, allowing $\omega_{a}$ and $\omega_{b}$ to be different.
\end{remark}
 
\begin{remark}
The above reasoning also leads to the following more general statement. Consider again, for
$C=a, b$,  actions $G_{C}\circlearrowleft M_{C}$ with corresponding homotopy moment maps $f^{C}$. Assume now that there is a manifold $B$ and $G_C$-equivariant submersions $\pi_C\colon M_C\to B$. Then the diagonal action of $G$ on the fiber product $M_a\times_B M_b=(\pi_a\times \pi_b)^{-1}(\Delta B)$, endowed with the pullback by the inclusion of $\omega_a\wedge \omega_b$, admits a homotopy moment map. 

The special case $M_a=M_b=B$ with $\pi_a=\pi_b=Id$ delivers $({M},{\omega}\wedge{\omega})$. Another interesting special case arises when $\pi_C\colon M_C\to B$ are principal $G_C$-bundles (in that case the action on $B$ is trivial).

\end{remark}
 
Making more explicit the formula for $i^{\ast}F \circ j$, we obtain:
\begin{prop}\label{prop:owedgeo}
 Let ${G}$ be a Lie group with Lie algebra $\g$, and fix an action of $G$ on an pre-$n$-plectic manifold $(M,\omega)$ with homotopy moment map $f\colon \g\to L_{\infty}(M,\omega)$, where $n$ is odd. 
 Then the $G$ action  on $({M}, {\omega}\wedge {\omega})$   has a homotopy moment map, with components ($k=1,\dots , 2n+1$)
 \begin{eqnarray*} 
 \mathfrak{g}^{\otimes k} &\to & L_{\infty}\left(M,\omega\wedge \omega \right)\,\nonumber\\ x^{1}\otimes \dots \otimes x^{k}  &\mapsto & 
2\sum_{m=1}^k \sum_{\sigma\in Sh_{m, k-m}}(-1)^{\sigma}c_{m,k-m}^a
f_{m}\left(x^{\sigma(1)}, \dots , x^{\sigma(m)}\right)\wedge \iota_{\sigma{(m+1)}, \dots ,\sigma(k)}\omega.
\end{eqnarray*}
 \end{prop}

\begin{remark}  The above double sum consist of $2^k-1$ summands.
\end{remark}
\begin{proof} Fix $k\ge 1$ and $x^1\wedge\dots\wedge x^k\in \wedge^k \g$. Notice that 
\begin{equation*}
j(x^1)\wedge\dots\wedge j(x^k)\in \wedge^k(\g\oplus \g)
\end{equation*}
 is the sum of $2^k$ monomials in a natural way. For instance, introducing the notation $j(x)=x_a\oplus x_b$, one has
$j(x^1)\wedge j(x^2)=x_a^1\wedge x_a^2
+ x_a^1\wedge x_b^2
+x_b^1\wedge x_a^2
+x_b^1\wedge x_b^2$.
Let $X$ denote   one of these monomials,  let $m$ be 
the number of  elements in $X$ decorated by the index ``$a$'', and $l:=k-m$.  
If $m=0$ or $l=0$, it is clear by Thm. \ref{thm:morphismproduct} that  $(i^*(F_k))(X)=F(X)\wedge \omega$.

Hence we consider only the case that $m,l\neq 0$. 
$X$ can be   written as $$(-1)^{\sigma}x_a^{\sigma(1)}\wedge\dots\wedge x_a^{\sigma(m)}\wedge x_b^{\sigma(m+1)}\wedge\dots\wedge x_b^{\sigma(k)}$$ for a unique $\sigma\in Sh_{m, l}$. 
 By Thm. \ref{thm:morphismproduct} we have
\begin{align}\label{eq:iotafk}
F_k(X)=(-1)^{\sigma}&\Big[c^{a}_{m,l}\,f_{m}\left(x^{\sigma(1)}_{a}, \dots , x^{\sigma(m)}_{a}\right)\wedge \iota_{\sigma(m+1), \dots ,\sigma(k)}\omega\\
& + c^{b}_{m,l}\,\iota_{\sigma(1), \dots ,\sigma(m)}\omega  \wedge f_{l}\left(x^{\sigma(m+1)}_{b}, \dots , x^{\sigma(k)}_{b}\right)\Big].\nonumber
\end{align}
Denote by $Y$ the monomial   obtained from $X$ interchanging each index ``$a$'' with the index ``$b$''. Notice that 
$Y$ can be  written as $$(-1)^{\tau}x_a^{\tau(1)}\wedge\dots\wedge x_a^{\tau(l)}\wedge x_b^{\tau(l+1)}\wedge\dots\wedge x_b^{\tau(k)}$$ for a unique $\tau\in Sh(l,m)$.
One can check that the first summand of
$ F_k (X)$ in eq. \eqref{eq:iotafk} agrees exactly with the second summand of $F_k(Y)$. Hence
\begin{align*}(i^*(F_k))(X+Y)=2&\Big[(-1)^{\sigma}c^{a}_{m,l}\,f_{m}\left(x^{\sigma(1)}, \dots , x^{\sigma(m)}\right)\wedge \iota_{\sigma(m+1), \dots ,\sigma(k)}\omega\\
+&\;\;(-1)^{\tau}c^{a}_{l,m}\,f_{m}\left(x^{\tau(1)}, \dots , x^{\tau(l)}\right)\wedge \iota_{\tau(l+1), \dots ,\tau(k)}\omega\Big].\nonumber
\end{align*}
Pairing two by two as above all the summands of $m(x^1)\wedge\dots\wedge m(x^k)$ and summing up, we see that $(i^*(F_k)\circ j)(x^1\wedge\dots\wedge x^k)$ equals the expression given in the statement of this proposition.
\end{proof}

Not all the homotopy  moment maps for  $({M}, {\omega}\wedge {\omega})$ arise from homotopy moment maps for $({M}, {\omega})$ as in Prop. \ref{prop:owedgeo}, as the following example shows.

\begin{ep}\label{ex:noowo} Consider the symplectic manifold
$M:=S^1\times S^1\times S^1\times \RR$ with canonical ``coordinates'' $\theta_1,\theta_2,\theta_3,x_4$, and 
  symplectic form
$\omega=
d\theta_1\wedge d\theta_2+d\theta_3\wedge dx_4$.
The action of the circle on $M$ with generator $\frac{\partial}{\partial \theta_1}$ is by symplectomorphisms, but does not admit a moment map since $d\theta_2$ is not exact.
 
 On the other hand $\omega\wedge\omega=
 2d\theta_1\wedge d\theta_2\wedge d\theta_3\wedge dx_4$
 is exact with invariant primitive (for instance, as primitive take $-2x_4d\theta_1\wedge d\theta_2\wedge d\theta_3$). Therefore by \cite[\S 8]{FRZ} there is a homotopy moment maps for $\omega\wedge\omega$,
constructed canonically using this primitive. 
\end{ep}


\subsection{Hyperk\"ahler manifolds} 


The results in this subsection are closely related to Martin Callies' results in \cite{Callies}.

\begin{defi}
A {\bf Hyperk\"ahler manifold}  is a  Riemannian manifold $(M,g)$ equipped with three complex structures $J_{i}: TM\to TM\, , i=1,2,3\, ,$   which satisfy the quaternionic relations $J^{2}_{i}= J_{1}J_{2}J_{3} = -1$ and are covariantly constant with respect to the Levi-Civita connection $\nabla$ associated to $g$, that is,  $\nabla J_{i} = 0\, , i=1,2,3\, .$ We say then that $\left( g, J_{1}, J_{2}, J_{3}\right)$ is a Hyperk\"ahler structure on $M$.
\end{defi}

\noindent
As a consequence of the definition of Hyperk\"ahler manifold, $M$ is also equipped with three symplectic two-forms $\omega_{i}\, , i=1,2,3\, ,$ as follows

\begin{equation*}
\omega_{i}\left(u,v\right) = g\left( J_{i}u,v\right)\,  ,\quad u,v \in\mathfrak{X}(M)\, ,\quad i=1,2,3\, .
\end{equation*}

\begin{remark}
\noindent
Notice that $\omega_{i}$ is non-degenerate as a consequence of $g$ and $J$ being non-degenerate and it is closed as a consequence of $J_{i}$ being covariantly constant. In fact, we have
$\nabla\omega_{i} = 0$ for $i=1,2,3$.

\noindent
If $a_{i}\in\mathbb{R}\, , i=1,2,3\, ,$ with $\sum^{3}_{i=1}a^2_{i}=1$, then $\sum^{3}_{i=1}a_{i} J_{i}$ is a complex structure un $M$ , and $g$ is K\"ahler respect to it, with K\"ahler form $\sum^{3}_{i=1}a_{i} \omega_{i}$. Hence, a Hyperk\"ahler manifold $M$ is equipped with a \emph{sphere} of complex structures and K\"ahler forms.

A Hyperk\"ahler manifold can be also characterized as a $4k$-dimensional (real) Riemannian manifold with Riemannian holonomy contained in   $Sp(k)$, where $k\geq 1$. Since $Sp(k)\subset SU(2k)$, every Hyperk\"ahler manifold is Calabi-Yau and Ricci-flat. Notice that the natural representation of $Sp(k)$ on   $\RR^{4k}$ preserves three complex structures $J_{i}\, , i=1,2,3\, ,$ that satisfy the quaternionic relations $J^{2}_{i}= J_{1}J_{2}J_{3} = -1$.
%
%
\end{remark}

It turns out that  $$\Omega:=\sum_{i=1}^3\omega_i\wedge \omega_i$$
is a 3-plectic form. 

The following Lemma follows immediately from Def. \ref{def:homotopymoment} using Eq. \eqref{main_eq_1} and \eqref{main_eq_2} (or alternatively from Prop. \ref{prop:doubleprimitive}).

 \begin{lemma}\label{lem:sum}
Suppose we are given an action of a Lie group $H$ on a manifold $N$ preserving pre-$n$-plectic forms $\Omega_1$ and $\Omega_2$, with homotopy moment maps $F^1$ and $F^2$ respectively. Then the action of $H$ on $(N,\Omega_1+\Omega_2)$ has homotopy moment map $F^1+F^2$.
\end{lemma}

\begin{prop}
Let $G$ be a Lie group acting on the Hyperk\"ahler manifold $M$. Assume that
$(M,\omega_i)$ admits an equivariant moment map $f^i$, for $i=1,2,3$.
Then the $G$ action on the 3-plectic manifold $(M,\Omega)$ admits
a homotopy moment map, constructed canonically out of $f^1,f^2,f^3$.
\end{prop}
\begin{proof}
Since $f^i$ is a moment map for $\omega_i$, Prop. \ref{prop:owedgeo}
provides a homotopy moment map $F^i$ for $\omega_i\wedge \omega_i$, for $i=1,2,3$. A homotopy moment map for $\Omega$ is then given by $F^1+F^2+F^3$,
by Lemma \ref{lem:sum}.
\end{proof}

Not all homotopy moment maps for $\Omega$ arise from moment maps for the $\omega_i$, as the following variation of Ex. \ref{ex:noowo} shows.
\begin{ep}
Consider the Hyperk\"ahler manifold $\RR^4$ with the canonical metric and the complex structures $J_1,J_2,J_3$ given by quaternionic multiplication by $i,j,k\in \mathbb{H}=\RR^4$.
Dividing by the lattice $\ZZ^3\times \{0\}$ we obtain a Hyperk\"ahler structure on $M:=S^1\times S^1\times S^1\times \RR$ (the product of the 3-torus with the real line),
on which we have induced ``coordinates'' $\theta_1,\theta_2,\theta_3,x_4$.
The  symplectic structures on $M$ associated to the distinguished complex structures are 
$$\omega_1=
d\theta_1\wedge d\theta_2+d\theta_3\wedge dx_4,\;\;\;\;
\omega_2=d\theta_1\wedge d\theta_3-d\theta_2\wedge dx_4,\;\;\;\;\;
\omega_3=d\theta_1\wedge dx_4+d\theta_2\wedge d\theta_3.$$
The action of the circle on $M$ with generator $\frac{\partial}{\partial \theta_1}$
preserves each $\omega_i$, however  $\omega_1$ and $\omega_2$
have no moment map for this action.
On the other hand, it is easily computed that $\Omega:=\sum_{i=1}^3\omega_i\wedge \omega_i=6d\theta_1\wedge d\theta_2\wedge d\theta_3\wedge dx_4$, and $\Omega$ admits a homotopy moment map as we explained in 
Ex. \ref{ex:noowo}. 
\end{ep}
 

\section{Embeddings of $\li$-algebras associated to closed differential forms}\label{section:emb}


Let $(M_C,\omega_C)$ be a pre-$n_C$-plectic manifold, $C=a\, , b$. 
We consider the  pre-$n_a+n_b+1$-plectic manifold
\begin{equation*}
\left( M\equiv M_a\times M_b,\omega\equiv\omega_a\wedge\omega_b\right)\, .
\end{equation*}
Being $(M_C,\omega_C)$ a pre-$n_{C}$-plectic manifold, it is equipped with a Lie $n_{C}$- algebra $L_{\infty}(M_{C},\omega_{C})$, 
 constructed exclusively out of $\omega_{C}$ and the de Rahm differential $d$. The purpose of this section is to find an $L_{\infty}$-morphism 
\begin{equation}\label{eq:Hto}
 H\colon L_{\infty}(M_a,\omega_a)\oplus L_{\infty}(M_b,\omega_b) \rightsquigarrow L_{\infty}(M_a\times M_b,\omega_a\wedge \omega_b)
\end{equation}
whose first component is an embedding. We will exhibit such a morphism in Thm. \ref{thm:embed}.  

 \begin{remark}
 As in the previous section, we will slightly abuse notation, denoting a differential form on $M_C$ and its pullback to $M_a\times M_b$, via  the canonical projection, by the same symbol. Similarly, given a vector field on $M_C$, we denote by the same symbol its horizontal lift to the product manifold $M_a\times M_b$.
 
 Further, we denote by $l^a$ and $l^b$ the multi-brackets of 
 $L_{\infty}(M_a,\omega_a)$ and  $L_{\infty}(M_b,\omega_b) $ respectively, and by $l$ 
 the multi-brackets of 
 $L_{\infty}(M,\omega)$. 
\end{remark}

\subsection{The construction of $H$ and its properties}

The source of $H$ is
$L_{\infty}(M_a,\omega_a)\oplus L_{\infty}(M_b,\omega_b)$, which, being a direct sum of 
$L_{\infty}$-algebras,
 is itself an $L_{\infty}$-algebra. We spell this out, assuming $n_{b}\geq n_{a}$. The underlying complex is
$$  
C^{\infty}(M_{b})\to \cdots\to C^{\infty}(M_{a})\oplus \Omega^{n_{b}-n_{a}}(M_{b})  
\to \cdots\to \Omega^{n_{a}-1}(M_{a})\oplus \Omega^{n_{b}-1}(M_{b}).
$$ 
Its  multibrackets $l^{ab}_{k}$ (for $k\geq 1$) are defined by
\begin{equation*}
l^{ab}_{k}\left(\alpha_{1} \oplus\beta_{1} ,\dots , \alpha_{k} \oplus\beta_{k}\right) = l^{a}_{k}\left(\alpha_{1},\dots , \alpha_{k}\right) \oplus l^{b}_k\left(\beta_{1},\dots ,\beta_{k}\right)
\end{equation*}
 where $\alpha_{1}\oplus\beta_{1},\dots , \alpha_{k}\oplus\beta_{k} \in L_{\infty}(M_a,\omega_a)\oplus L_{\infty}(M_b,\omega_b)$.
 Notice that $L_{\infty}(M_a,\omega_a)\oplus L_{\infty}(M_b,\omega_b)$ is a Lie $N$-algebra, where $N:=Max\{n_a,n_b\}$, while  $L_{\infty}(M,\omega)$ - the target of $H$ -  is a Lie $(n_a+n_b+1)$-algebra. \\


We now argue that there is a natural candidate for the  first component of an  $L_{\infty}$-morphism as in \eqref{eq:Hto}. 
Given $\alpha \in \ham{n_a-1}{M_a}$ and $\beta \in \ham{n_b-1}{M_b}$, take Hamiltonian vector fields $X_{\alpha }$ and $X_{\beta}$ for them, and 
consider  $X_{\alpha } +   X_{\beta}$ on $M_a\times M_b$. It is again a Hamiltonian vector field, since
\begin{equation*}
\iota_{(X_{\alpha} +   X_{\beta})}\omega = - d\left[  \alpha \wedge \omega_b +  \omega_{a}\wedge \beta\right].
\end{equation*}
\noindent
Hence there is a well-defined map 
\begin{eqnarray*}
h: \Omega^{n_{a}-1}_{\mathrm{Ham}}\left(M_{a}\right)\oplus \Omega^{n_{b}-1}_{\mathrm{Ham}}\left(M_{b}\right) &\to &  \Omega^{n_{a}+n_{b}}_{\mathrm{Ham}}\left(M\right)\nonumber\\
\alpha \oplus \beta &\mapsto &  \alpha \wedge \omega_{b} +  \omega_{a}\wedge \beta\, .
\end{eqnarray*}

\noindent
Endow $\Omega^{n_{a}-1}_{\mathrm{Ham}}\left(M_{a}\right)\oplus \Omega^{n_{b}-1}_{\mathrm{Ham}}\left(M_{b}\right)$ with the bracket $l^{ab}_2$, i.e., the
sum of the binary brackets 
$l_2^a$ and $l_2^b$ on the two factors. Denoting all   binary brackets by $\{\cdot,\cdot\}$ to ease the notation, we have

\begin{equation*}
h\Big(\Big\{\alpha_{1}\oplus \beta_{1},\alpha_{2}\oplus \beta_{2}\Big\}\Big) =
\Big\{ h(\alpha_{1}\oplus\beta_{1})\;,\; h(\alpha_{2}\oplus\beta_{2})\Big\} + (-1)^{n_a} d\left[ \alpha_{1}\wedge  d\beta_{2} - \alpha_{2}\wedge  d\beta_{1}\right]\, . 
\end{equation*}

\noindent
That is, $h$ does not preserve the binary brackets on the nose, but just  up to an exact term. This a characteristic feature of the  {first component} of an $L_{\infty}$-morphism. Indeed, in  Thm. \ref{thm:embed} we extend $h$ to an
$L_{\infty}$-morphism from $L_{\infty}(M_a,\omega_a)\oplus L_{\infty}(M_b,\omega_b)$ to $L_{\infty}(M,\omega)$. 
  The concrete expression   of the  $L_{\infty}$-morphism 
is  motivated by the results of Section \ref{sec:Momentmapproduct} and in particular by Theorem \ref{thm:morphismproduct}.\\

We will use the square brackets  introduced in Def. \ref{def:square}, for $C= a,b$. Recall that $[\dots]^C_k$ is defined for all $k\ge 0$, and that it vanishes unless all entries have degree zero (i.e., are Hamiltonian forms). Recall also that 
$\left[\,1\right]^{C}_{0} = -\omega_{C}$ and that for $k\ge 1$,
by Remark \ref{prop:iotaomega},
  \begin{equation*}
[\alphadk{k}]^C_k=\{\alphadk{k}\}_{C}-\delta_{k,1}d_{C}\alpha_1\, ,\quad C= a, b\, ,
\end{equation*}
where $\{\alphadk{k}\}_{C}$ is the $k$-bracket of $L_{\infty}(M_{C},\omega_{C})$ and $d_{C}$ is the de Rahm differential on $M_{C}$.  

\begin{thm}\label{thm:embed}
Let $ (M_C,\omega_C)$ be pre-$n_{C}$-plectic manifolds. There is an $\li$-morphism 

\begin{equation*}
H\colon L_{\infty}(M_a,\omega_a)\oplus L_{\infty}(M_b,\omega_b) \rightsquigarrow L_{\infty}(M_a\times M_b,\omega_a\wedge \omega_b)
\end{equation*}
whose first component is injective.
The components of $H$ will be denoted by $H_l$ ($l\ge 1$). They are determined by graded skew-symmetry and the requirement that
  \begin{center}
\fbox{
\begin{Beqnarray*}
H_{k+m}(\alpha_1,\dots,\alpha_k,\beta_1,\dots,\beta_m)=&
 t^{a}_{m,|\alpha_{1}|}\delta_{k,1}\alpha_1\wedge [\beta_1,\dots,\beta_m]^{b}_{m}\\ +& \, t^{b}_{k,|\beta_{1}|}\delta_{m,1}[\alpha_1,\dots,\alpha_k]^{a}_{k}\wedge\beta_1\, ,
\end{Beqnarray*}
}
\end{center}
where $k+m\ge 1$, $\alpha_1,\dots,\alpha_k\in L_{\infty}(M_a,\omega_a)$, $\beta_1,\dots,\beta_m\in L_{\infty}(M_b,\omega_b)$, $\left[\,1\;\right]^{C}_{0} = -\omega_{C}$ and 
the coefficients are, for all $i\leq 0$:
\begin{align}
\label{eq:s}
t^{a}_{m,i} &= -\frac{1}{2}\left(-1\right)^{m(n_{a}+1+i)}\, , \qquad m\geq 1\\
t^{b}_{k,i} &= -\frac{1}{2}\left(-1\right)^{i(n_{a}+1)+k}\,\,\, , \qquad k\geq 1\,\nonumber 
\end{align}
and
$$t^{a}_{0,i} = -1,\;\;\;\;\;\;\; t^{b}_{0,i} = -\left(-1\right)^{i(n_{a}+1)}.$$

\noindent
\newline
Above, $\delta$ denotes the Kronecker delta, and  $|\alpha_{1}|$ refers to the degree\footnote{This  differs by $n_a-1$ from the degree of $\alpha_1$ as a differential form.} of $\alpha_1$  as an element of $L_{\infty}(M_a,\omega_a)$.
\end{thm}

\begin{remark}\label{rem:onlyone}
Notice that $H$, applied to a family of elements lying in $(L_{\infty}(M_a,\omega_a)\oplus\{0\})\cup (\{0\}\oplus L_{\infty}(M_b,\omega_b))$, vanishes unless: either exactly one element is of the form $\alpha\oplus 0$ and the remaining elements have degree zero, or exactly one element is of the form $0\oplus \beta$ and the remaining elements have degree zero.
\end{remark}

\begin{remark}
The first component $H_1$ is clearly injective for it is given by
$$H_1(\alpha)=\alpha\wedge\omega_b \;\;\;\;\; \text{ and }\;\;\;\;\;
H_1(\beta)=
\left(-1\right)^{|\beta|(n_{a}+1)}\omega_a\wedge\beta,$$ where $\alpha \in L_{\infty}(M_a,\omega_a)$ and $\beta\in L_{\infty}(M_b,\omega_b)$.

The restriction of $H_1$ to $L_{\infty}(M_a,\omega_a)\oplus\{0\}$ is a strict morphism. This can be seen using Remark \ref{rem:onlyone}, since the higher components of $H$ vanish if all entries lie in $L_{\infty}(M_a,\omega_a)\oplus\{0\}$, or alternatively it can be seen directly using Lemma \ref{lem:square} below. The same holds for the restriction of $H$ to $\{0\}\oplus L_{\infty}(M_b,\omega_b)$.
\end{remark}

\begin{remark}
{Recall that the composition $\psi \circ \phi$ of two $L_{\infty}$-morphisms is given  by
$(\psi \circ \phi)_{k}=\sum_{l=1}^k\sum_{k_1+\dots +k_l=k}\pm\psi_l\circ (\phi_{k_1}\otimes\dots\otimes \phi_{k_l})$.} Possibly up to signs, the $\li$-morphism $H$ given  Thm. \ref{thm:embed} has the following property:
for any action of a Lie group $G_C$ on $(M_C,\omega_C)$ with homotopy  moment map $f^C$ ($C=a,b$), one has $$F=H\circ (f^a\oplus f^b),$$ where $F$ is the homotopy moment map constructed in Thm. \ref{thm:morphismproduct} out of $f^a$ and $f^b$.
In other words, the diagram \eqref{diag:productmoment} commutes. 
\end{remark}

\begin{ep}
Let $n_a=n_b=1$. That is, $\left(M_{a},\omega_{a}\right)$ and $\left(M_{b},\omega_{b}\right)$ are pre-symplectic manifolds, and so  $\left(M,\omega\right)$ is a pre-$3$-plectic manifold. Consequently, the cochain complex $L$ underlying the Lie 3-algebra  $L_{\infty}\left(M,\omega\right)$ is 
\begin{equation*}
 C^{\infty}\left(M\right)\to\Omega^{1}\left(M\right)\to\Omega^{2}\left(M\right)\to \Omega^{3}_{\mathrm{Ham}}\left(M\right)\,.
\end{equation*}

\noindent
On the other hand, $L_{\infty}\left(M_{a},\omega_{a}\right)\oplus L_{\infty}\left(M_{b},\omega_{b}\right) = C^{\infty}\left(M_{a}\right)\oplus C^{\infty}\left(M_{b}\right)$  is just a Lie-algebra. The higher components of the $L_{\infty}$-embedding of theorem \ref{thm:embed} read 
 
\begin{equation*}
\mathrm{H}_{2}\left( f_{a}\oplus f_{b} , g_{a}\oplus g_{b}\right) =\frac{1}{2} \left( f_{a}\wedge dg_{b} - df_{a}\wedge g_{b} - g_{a}\wedge df_{b} + dg_{a}\wedge f_{b}\right)\, ,
\end{equation*}

\begin{eqnarray*}
\mathrm{H}_{3}\left( f_{a}\oplus f_{b} , g_{a}\oplus g_{b}, h_{a}\oplus h_{b}\right) = \frac{1}{2}\left(f_{a}\left\{g_{b},h_{b}\right\}_{2} + f_{b}\left\{g_{a},h_{a}\right\}_{2} - g_{a}\left\{f_{b},h_{b}\right\}_{2} \right.\\ \left. - g_{b}\left\{f_{a},h_{a}\right\}_{2} + h_{a}\left\{f_{b},g_{b}\right\}_{2} + h_{b}\left\{f_{a},g_{a}\right\}_{2}\right)\, ,
\end{eqnarray*}

\noindent
for all $f_{C}, g_{C}, h_{C} \in C^{\infty}\left(M_{C}\right)\, ,\,\, C =a,b\, .$ Notice that since $L_{\infty}\left(M_{a},\omega_{a}\right)\oplus L_{\infty}\left(M_{b},\omega_{b}\right)$ is a   Lie algebra, we can use formulae \eqref{main_eq_1} and \eqref{main_eq_2} to double-check that $H$ is indeed an $L_{\infty}$-morphism.  
\end{ep} 

\subsection{The proof}

We now turn to the  proof of Thm. \ref{thm:embed}. We will use repeatedly the following Lemma.

\begin{lemma}\label{lem:square}
For all $\alpha_{1},\ldots,\alpha_{k}\in L_{\infty}(M_a,\omega_a)$ and $\beta_{1},\ldots,\beta_{m}\in L_{\infty}(M_b,\omega_b)$, where $k,m \ge 0$ and $k+m\ge 1$, we have
$$[\alpha_1\omega_b,\dots,\alpha_k\omega_b,\omega_a\beta_1,\dots,\omega_a\beta_m]_{k+m}=
-(-1)^{m(n_{a}+1)}[\alpha_1,\dots,\alpha_k]_{k}\wedge[\beta_1,\dots,\beta_m]_{m}\, .$$
\end{lemma}

\begin{proof} 
We may assume that all the $\alpha$ and $\beta$ have degree zero, for otherwise the equation is trivially satisfied. It is straightforward to verify that 
the Hamiltonian vector field of $\alpha\omega_b$ (w.r.t $\omega_a\wedge \omega_b$) equals the Hamiltonian vector field $X_{\alpha}$ of $\alpha$ (w.r.t $\omega_a$), and the exactly analogous statement holds for  $\omega_a\beta$. The statement of the lemma follows from
$$\iota(X_{\alpha_1}\wedge X_{\alpha_2}\wedge \dots \wedge X_{\beta_m})(\omega_a\wedge \omega_b)=(-1)^{m(n_{a}+1-k)}
\iota(X_{\alpha_1}\wedge   \dots \wedge X_{\alpha_k})\omega_a\wedge
\iota(X_{\beta_1}\wedge   \dots \wedge X_{\beta_m})\omega_b$$
together with the identity
$\vs(k)\vs(m)\vs(k+m)=-(-1)^{km}$. 
\end{proof}

According to the conditions that an $L_{\infty}$-morphism has to obey (see for instance \cite[Def. 2.4]{WurzRyvkinMomaps}), we have to check  that the following relation holds for all $N\in\mathbb{N}_{>0}$ and for all $\vec{x}=(x_{1},\ldots,x_{N})\in \left(L_{\infty}\left(M_{a},\omega_{a}\right)\oplus L_{\infty}\left(M_{b},\omega_{b}\right)\right)^{\otimes N}$: 
\begin{align}
\label{eq:Linfinityconditions}
& \sum_{i+j=N+1}(-1)^{i(j-1)}\sum_{\sigma\in \mathrm{Sh}_{i,j-1}}(-1)^{\sigma}\epsilon(\sigma, \vec{x}
)\,H_{j}\left(l^{ab}_{i}\left( x_{\sigma(1)},\ldots,x_{\sigma(i)}\right) ,x_{\sigma(i+1)}%
,\ldots,x_{\sigma(N)}\right)\\
  = &\sum_{\ell=1}^{N}\sum_{\substack{N_{1}+\cdots+N_{\ell}=N\\N_{1}%
\leq\cdots\leq N_{\ell}}}(-1)^{\gamma(\ell,\vec{N})}\sum_{\sigma\in Sh_{N_{1},\ldots,N_{\ell}}^{<}}%
(-1)^{\sigma}\epsilon(\sigma, \vec{x}
)
\epsilon\left(\rho,\vec{H}\right)\, \nonumber \\
&\;\;\;
l_{\ell}\left( H_{N_{1}}(x_{\sigma(1)},\ldots,x_{\sigma
(N_{1})}),\ldots, H_{N_{\ell}}(x_{\sigma(N-N_{\ell}+1)},\ldots,x_{\sigma
(N)})\right).\, \nonumber
\end{align}
Here

\begin{itemize}

\item 
$\gamma(\ell,\vec{N})\equiv \frac{\ell(\ell -1)}{2} + N_{1} (\ell-1)+ N_{2} (\ell-2) +\dots + N_{\ell-1}$.

\item $Sh_{N_{1},\ldots,N_{\ell}}^{<}\subset Sh_{N_{1},\ldots,N_{\ell}}$ is the set of $(N_{1},\ldots,N_{\ell})$-unshuffles such that
\[
\sigma(N_{1}+\cdots+N_{i-1}+1)<\sigma(N_{1}+\cdots+N_{i-1}+N_{i}+1)\text{\quad
whenever }N_{i}=N_{i+1}.
\]
\item $\vec{H} = \left(H_{N_{1}},\ldots,H_{N_{\ell}},x_{\sigma(1)},\ldots,x_{\sigma(N)}\right)$ and $\rho$ is the permutation of $\left\{1,\ldots,\ell+N\right\}$ sending $\vec{H}$ to 
$
\left(H_{N_{1}},x_{\sigma(1)},\ldots,x_{\sigma(N_{1})},\ldots,H_{N_{\ell}},x_{\sigma(N-N_{\ell}+1)},\ldots,x_{\sigma(N)}\right)\, .
$
\end{itemize}

\noindent
As usual, $(-1)^{\sigma}$ denotes the  sign of the permutation $\sigma$ and $\epsilon(\sigma, \vec{x})$ denotes the Koszul  sign.
 
\begin{remark}\label{rem:signsLRHS}
Notice that on the l.h.s. of eq. \eqref{eq:Linfinityconditions}, the sign of the summand corresponding to $i=N,j=1$ is $+1$ (since the only permutation appearing is the identity).

On the r.h.s., the sign of the summand corresponding to $l=N$ is   $+1$. Indeed $N_1=\dots=N_l=1$, so that   $\gamma(\ell,\vec{N})=+1$, $\sigma=id$, and all $H_{N_i}$ have degree zero. Further, the sign of the summand corresponding to $\ell=1$ is also $+1$, since $\gamma(1,\vec{N})=+1$, $\sigma=id$ and $\rho=id$.
\end{remark}

\begin{proof}[Proof of Thm. \ref{thm:embed}]
Let $C = a, b$. We first check that $H_{j}$ has degree $1-j$. For $j=1$ this is clear. For $j=k+m\ge 2$, we use that $[\dots]^C_m$, as an operation on $L_{\infty}(M_C,\omega_C)$, has degree $2-m$. Hence, for instance, if the   elements $\alpha_1,\beta_1,\dots,\beta_m$ all have degree zero, then $H_{1+m}(\alpha_1,\beta_1,\dots,\beta_m)=\pm\frac{1}{2}\alpha_1[\beta_1,\dots,\beta_m]^{b}_{m}$ is the product of a $n_a-1$ and $(n_b-1)+(2-m)$ form, that is, a $n_a+n_b-m$ form, which therefore is an element of $L_{\infty}(M_a\times M_b,\omega_a\wedge \omega_b)$ of degree $-m = 1-(1+m) = 1-j\, .$

The rest of the proof is devoted to checking that $H$ is an $L_{\infty}$-morphism.
Our strategy is as follows.
 We  propose an educated ansatz for $H$
 depending on some arbitrary parameters and then we will impose on it the $L_{\infty}$-morphism conditions \eqref{eq:Linfinityconditions}. Equations \eqref{eq:Linfinityconditions}  will turn out to be an over-determined system of equations for the parameters of the ansatz, and we will show that a solution is given by eq. \eqref{eq:s}.

The ansatz is the following:  for the first component of $H$,
 $$H_1(\alpha)=s^{a}_{0,|\alpha|}\alpha\wedge(-\omega_b),\;\;\;\;\;\;\; H_1(\beta)=s^{b}_{0,|\beta|} (-\omega_a)\wedge\beta.$$
For the higher components of $H$, i.e. for   $k+m\ge 2$,
  $H_{k+m}(\alpha_1,\dots,\alpha_k,\beta_1,\dots,\beta_m)$
equals 

\begin{equation}
\label{eq:Hgeneralansatz}
\frac{s^{a}_{m,|\alpha_{1}|}}{2}\delta_{k,1}\alpha_1\wedge[\beta_1,\dots,\beta_m]^{b}_{m} + \frac{s^{b}_{k,|\beta_{1}|}}{2}\delta_{m,1}[\alpha_1,\dots,\alpha_k]^{a}_{k}\wedge\beta_1\, ,
\end{equation}
where
 $\alpha_{1},\ldots,\alpha_{k}\in L_{\infty}(M_a,\omega_a)$ and $\beta_{1},\ldots,\beta_{m}\in L_{\infty}(M_b,\omega_b)$ are homogeneous elements of their respective graded spaces. Here $s^{a}_{m,|\alpha_{1}|}$ depends on the number of $\beta$'s and the degree of $\alpha_{1}$. It cannot depend on the number of $\alpha$'s since if there is more than one the corresponding term in \eqref{eq:Hgeneralansatz} is zero, and it cannot depend on the degree of the $\beta$'s since if $\left|\beta_{1}\otimes\cdots\otimes\beta_{m}\right|<0$ then the corresponding term in  \eqref{eq:Hgeneralansatz} is again zero. A similar discussion applies to $s^{b}_{k,|\beta_{1}|}$. 
 
We now  apply condition \eqref{eq:Linfinityconditions} to our ansatz for $H$ and elements $\alpha_1,\dots,\alpha_k,\beta_1,\dots,\beta_m$.
 We are going to consider six different cases depending on $k$ and $m$, namely $\left\{ k\geq 1, m = 0\right\}$, $\left\{ k = 0, m \geq  1\right\}$,  $\left\{ k = 1, m = 1\right\}$, $\left\{ k > 1, m > 1\right\}$, $\left\{ k = 1, m > 1\right\}$ and $\left\{ k > 1, m = 1\right\}$. We will use repeatedly Remark \ref{rem:onlyone} and the fact that for $i\ge 2$ the multibrackets $l_i$ vanish unless all entries have degree zero.


\bigskip
 \noindent{\bf Case $\mathbf{\left\{ k\geq 1, m = 0\right\}}$.}
 
This case will allow us to calculate $s^{a}_{0,i},\, i\leq 0$. The condition \eqref{eq:Linfinityconditions} evaluated on $\alpha_{1},\ldots,\alpha_{k}\in L_{\infty}(M_a,\omega_a)$ reads
\begin{equation}
\label{eq:conditionk0}
H_{1}\left(l^{a}_{k}\left(\alpha_{1},\ldots,\alpha_{k}\right)\right) = l_{k}\left(H_{1}\left(\alpha_{1}\right),\ldots,H_{1}\left(\alpha_{k}\right)\right)\, ,
\end{equation}
as one sees using Rem. \ref{rem:onlyone}, together with Remark \ref{rem:signsLRHS} to determine the signs.

\noindent
Using now that  

\begin{eqnarray*}
H_{1}\left(l^{a}_{k}\left(\alpha_{1},\ldots,\alpha_{k}\right)\right) &=& -s^{a}_{0,2-k + |\alpha|}\, l^{a}_{k}\left(\alpha_{1},\ldots,\alpha_{k}\right)\wedge \omega_{b} \, , \qquad \\l_{k}\left(H_{1}\left(\alpha_{1}\right),\ldots,H_{1}\left(\alpha_{k}\right)\right) &=& (-s^{a}_{0,|\alpha_1|} )\dots  (-s^{a}_{0,|\alpha_k|})\, l^{a}_{k}\left(\alpha_{1},\ldots,\alpha_{k}\right)\wedge \omega_{b}\, ,
\end{eqnarray*}

\noindent
where $\left|\alpha\right| = \left|\alpha_{1}\otimes\cdots\otimes\alpha_{k}\right|$ and using  Lemma \ref{lem:square} in the second equation when $k\ge 2$, we conclude that 
%
%
we can choose $s^{a}_{0,i} = -1$ for all $i\leq 0$.

\bigskip
 \noindent {\bf Case $\mathbf{\left\{ k = 0, m \geq 1\right\}}$.}

 This case will allow as to calculate $s^{b}_{0,i},\, i\leq 0$. The condition \eqref{eq:Linfinityconditions} evaluated on $\beta_{1},\ldots,\beta_{m}\in L_{\infty}(M_b,\omega_b)$, similarly to the case above, reads 

\begin{equation}
\label{eq:condition0m}
H_{1}\left(l^{b}_{m}\left(\beta_{1},\ldots,\beta_{m}\right)\right) = l_{m}\left(H_{1}\left(\beta_{1}\right),\ldots,H_{1}\left(\beta_{m}\right)\right)\, .
\end{equation}

\noindent
Using now that

\begin{eqnarray*}
H_{1}\left(l^{b}_{m}\left(\beta_{1},\ldots,\beta_{m}\right)\right) &=& -s^{b}_{0,2-m + |\beta|}\, \omega_{a}\wedge l^{b}_{m}\left(\beta_{1},\ldots,\beta_{m}\right)\, , \qquad \\l_{m}\left(H_{1}\left(\beta_{1}\right),\ldots,H_{1}\left(\beta_{m}\right)\right) &=& (-s^{b}_{0,|\beta_1|} )\dots  (-s^{b}_{0,|\beta_m|})
(-1)^{m(n_{a}+1)}\, \omega_{a}\wedge l^{b}_{m}\left(\beta_{1},\ldots,\beta_{m}\right)\, ,
\end{eqnarray*}

\noindent
where $\left|\beta\right| = \left|\beta_{1}\otimes\cdots\otimes\beta_{m}\right|$ and using  Lemma \ref{lem:square} in the second equation when $m\ge 2$, we conclude (taking $m=1$) that equation \eqref{eq:conditionk0} implies
$s^{b}_{0,1 +|\beta|} = (-1)^{(n_{a}+1)} s^{b}_{0,|\beta|}$
and therefore 
$s^{b}_{0,i} = \left(-1\right)^{i(n_{a}+1)}s^{b}_{0,0}\, ,\quad i\leq 0$.
Plugging this into into eq. \eqref{eq:condition0m} it can be easily verified that 
eq. \eqref{eq:condition0m} is solved by
 \begin{equation*}
s^{b}_{0,i} = -\left(-1\right)^{i(n_{a}+1)}\, ,\quad i\leq 0\, .
\end{equation*}


\bigskip
 \noindent{\bf Case $\mathbf{\left\{ k = 1, m = 1\right\}}$.}

 This case will allow as to find $s^{a}_{1,i}$ and $s^{b}_{1,i}$ for $i\leq 0$. The condition \eqref{eq:Linfinityconditions} evaluated on $\alpha,\beta$, where $\alpha\in L_{\infty}(M_a,\omega_a)$ and $\beta\in L_{\infty}(M_b,\omega_b)$, reads 

\begin{equation}
\label{eq:condition11}
-H_{2}\left(l^{a}_{1}(\alpha),\beta\right) -(-1)^{|\alpha|} H_{2}\left(\alpha,l^{b}_{1}(\beta)\right) = l_{1}\left(H_{2}(\alpha,\beta)\right)+l_{2}\left(H_{1}(\alpha),H_{1}(\beta)\right)\, .
\end{equation}
(The l.h.s. corresponds to the summand $i=1,j=2$ in \eqref{eq:Linfinityconditions}, and the signs for the r.h.s. follow from Remark \ref{rem:signsLRHS}.)
Recall that by ansatz \eqref{eq:Hgeneralansatz}, for all $A\in L_{\infty}(M_a,\omega_a)$ and $B\in L_{\infty}(M_b,\omega_b)$ we have
\begin{equation*}
H_{2}(A,B) = \frac{s^{a}_{1,|A|}}{2} A \wedge\left[B\right]^{b}_{1} + \frac{s^{b}_{1,|B|}}{2} \left[A\right]^{a}_{1} \wedge B.
\end{equation*}

In order to solve equation \eqref{eq:condition11} we have to analyze the different cases in terms of the degree of $\alpha$ and $\beta$. If $|\alpha| = |\beta| = 0$ the l.h.s. of \eqref{eq:condition11} is zero 
%
%
while the r.h.s. is
\begin{equation*}
-\frac{s^{a}_{1,0}}{2} d\alpha\wedge d\beta -(-1)^{n_{a}} \frac{s^{b}_{1,0}}{2} d\alpha\wedge d\beta  +(-1)^{n_{a}} d\alpha\wedge d\beta \,,
\end{equation*}

\noindent
as one sees using Lemma \ref{lem:square}. Hence we can take
 \begin{equation}
\label{eq:sasb10}
s^{a}_{1,0}  = (-1)^{n_{a}}\,,\;\;\;  s^{b}_{1,0}=1.
\end{equation}

\noindent
Now, if $|\alpha| = 0$ and $|\beta| < 0$  the first and fourth term in equation \eqref{eq:condition11} vanish, and that equation translates into

\begin{equation*}
-\frac{s^{b}_{1,|\beta|+1}}{2} \left[\alpha\right]^{a}_{1}\wedge l^{b}_{1}(\beta) = (-1)^{n_{a}}\frac{s^{b}_{1,|\beta|}}{2} \left[\alpha\right]^{a}_{1}\wedge l^{b}_{1}(\beta)\, ,
\end{equation*}

\noindent
implying that 
$s^{b}_{1,|\beta|+1} = (-1)^{n_{a}+1} s^{b}_{1,|\beta|}\, .$
Together with equation \eqref{eq:sasb10} this implies finally that 

\begin{equation*}
s^{b}_{1,i} = \left(-1\right)^{i(n_{a}+1)}\, , \qquad i\leq 0\, .
\end{equation*}

\noindent
By means of a completely analogous calculation for the case $|\alpha| < 0$ and $|\beta| = 0$ we obtain $s^{a}_{1,|\alpha|+1} = - s^{a}_{1,|\alpha|}$, so we can choose

\begin{equation*}
s^{a}_{1,i} = \left(-1\right)^{n_{a}+i}\, , \qquad i\leq 0\, .
\end{equation*}

\noindent
Lastly, the case $|\alpha|<0$ and $|\beta|<0$ is trivial since both sides of equation \eqref{eq:condition11} vanish.

\bigskip
 \noindent{\bf Case $\mathbf{\left\{ k > 1, m > 1\right\}}$.}

 This case will allow as to find $s^{a}_{k,i}$ and $s^{b}_{m,i}$ for $i\leq 0$ and $k,m>1$. The condition \eqref{eq:Linfinityconditions} evaluated on $(\alpha_{1},\ldots,\alpha_{k},\beta_{1},\ldots,\beta_{m})$, where $\alpha_{1},\ldots,\alpha_{k}\in L_{\infty}(M_a,\omega_a)$ and $\beta_{1},\ldots,\beta_{m}\in L_{\infty}(M_b,\omega_b)$, reduces to 

\begin{align}
\label{eq:conditionm1m1}
& (-1)^{km} H_{m+1}\left(l^{a}_{k}(\alpha_{1},\ldots,\alpha_{k}),\beta_{1},\ldots,\beta_{m}\right) +(-1)^k H_{k+1}\left(\alpha_{1},\ldots,\alpha_{k},l^{b}_{m}(\beta_{1},\ldots,\beta_{m})\right)\\  =  &l_{k+m}\left(H_{1}(\alpha_{1}),\ldots,H_{1}(\alpha_{k}),H_{1}(\beta_{1}),\ldots,H_{1}(\beta_{m})\right)\, ,\nonumber
\end{align}

\noindent
where in the l.h.s. only the summands corresponding to $i=k$ and $i=m$
appear
by  Rem. \ref{rem:onlyone}, and
for the r.h.s. we use Remark \ref{rem:signsLRHS} to determine the signs (a term involving $l_1$ does not appear, again due to Rem. \ref{rem:onlyone}).

From Def. \ref{def:multisymLinfinity}  it can be seen that equation \eqref{eq:conditionm1m1} is only non-trivial if\footnote{The fact that necessarily  $|\alpha|= 0$
was already used to determine the sign of the second term on the l.h.s. above.}
 $|\alpha| = |\beta| = 0$. Therefore, we will assume henceforth that this is the case. The two terms on the l.h.s. of equation \eqref{eq:conditionm1m1} can be written as follows:
\begin{eqnarray*}
H_{m+1}\left(l^{a}_{k}(\alpha_{1},\ldots,\alpha_{k}),\beta_{1},\ldots,\beta_{m}\right) = \frac{s^{a}_{m,2-k}}{2} \left[\alpha_{1},\ldots,\alpha_{k}\right]^{a}_{k}\wedge\left[\beta_{1},\ldots,\beta_{m}\right]^{b}_{m}\, , \\
H_{k+1}\left(\alpha_{1},\ldots,\alpha_{k},l^{b}_{m}(\beta_{1},\ldots,\beta_{m})\right) = \frac{s^{b}_{k,2-m}}{2} \left[\alpha_{1},\ldots,\alpha_{k}\right]^{a}_{k}\wedge\left[\beta_{1},\ldots,\beta_{m}\right]^{b}_{m}\, .
\end{eqnarray*}

\noindent
By Lemma \ref{lem:square}, the r.h.s. of equation \eqref{eq:conditionm1m1} can be written as

\begin{eqnarray*}
l_{k+m}\left(H_{1}(\alpha_{1}),\ldots,H_{1}(\alpha_{k}),H_{1}(\beta_{1}),\ldots,H_{1}(\beta_{m})\right) = -(-1)^{m(n_{a}+1)}\left[\alpha_{1},\ldots,\alpha_{k}\right]^{a}_{k}\wedge\left[\beta_{1},\ldots,\beta_{m}\right]^{b}_{m}\, .
\end{eqnarray*}

\noindent
From the last three equations
we obtain

\begin{equation*}
(-1)^{km}\frac{s^{a}_{m,2-k}}{2} +(-1)^k \frac{s^{b}_{k,2-m}}{2} = -(-1)^{m(n_{a}+1)}\, ,
\end{equation*}

\noindent
which is solved by  
\begin{equation*}
s^{a}_{m,i} = -(-1)^{m (n_{a}+i+1)} \, , \qquad s^{b}_{k,i} = -(-1)^{i(n_{a} {+}1)+k}\, ,\qquad m,k>1\,,\,\, i\leq 0\, .
\end{equation*}


\bigskip
 \noindent{\bf Cases $\mathbf{\left\{ k = 1, m > 1\right\}}$ and $\mathbf{\left\{ k > 1, m = 1\right\}}$.}

 Notice that this point we have already explicitly solved all the parameters $s^{a}_{m,i}$ and $s^{b}_{k,i}$ for all $k,m\geq 0$ and $i\leq 0$. Although this was obtained by separately analyzing different cases given by different values of $k$ and $m$, the result be summarized in a single formula, namely

\begin{equation}
\label{eq:sformula}
s^{a}_{m,i} = -(-1)^{m (n_{a}+i+1)} \, , \qquad s^{b}_{k,i} = -(-1)^{i(n_{a} {+}1)+k}\, , \qquad m,k\geq 0\, ,\,\,i\leq 0\, .
\end{equation}

\noindent
However, there remain two cases to be solved, namely $\left\{ k = 1, m > 1\right\}$ and $\left\{ k > 1, m = 1\right\}$. Notice that we do not have any parameter left to be fixed, so checking those cases is really a constraint.

We consider first the case $\left\{ k = 1, m > 1\right\}$. At first, we also assume $m>2$.
The condition \eqref{eq:Linfinityconditions} evaluated on $(\alpha,\beta_{1},\ldots,\beta_{m})$ reads
\begin{align}\label{eq:long}
&(-1)^m H_{m+1}\left(l^{a}_{1}(\alpha),\beta_{1},\ldots,\beta_{m}\right) 
+ \sum_{1\le p< q\le m} (-1)^{p+q}H_{m}\left(\alpha, l^{b}_{2}(\beta_p,\beta_q), \beta_1,\ldots,
\widehat{\beta_p},\dots,\widehat{\beta_q},\dots
\beta_{m})\right)
\\ 
&+ H_2((l^{b}_{m}(\beta_{1},\ldots,\beta_{m}),\alpha)\nonumber\\
  = \;\; & l_{m+1}\left(H_{1}(\alpha),H_{1}(\beta_{1}),\ldots,H_{1}(\beta_{m})\right)
+l_1(H_{m+1}(\alpha,\beta_{1},\ldots,\beta_{m}))
\, .\nonumber
\end{align}
(On the l.h.s. the first term corresponds to $i=1$ in eq. \eqref{eq:Linfinityconditions},
 the second to $i=2$, and the third to $i=m$; not other values of $i$ contribute by Remark \ref{rem:onlyone}. On the r.h.s. only the terms corresponding to $l_{m+1}$ and $l_1$ appear since   
the multibrackets of $L_{\infty}(M_a\times M_b)$ with two or more entries vanish  unless all the entries have degree zero, and the signs are given by Remark \ref{rem:signsLRHS}.) We may assume\footnote{This assumption was already used to determine the sign of the second term on the l.h.s. above.}   $|\beta_1|=\dots=|\beta_m|=0$, for otherwise both sides of the above equation vanish by
Remark \ref{rem:onlyone}.

The first term on the l.h.s. of eq. \eqref{eq:long} reads 
\begin{equation}\label{eqone}
(-1)^m \frac{s^a_{m,|\alpha|+1}}{2} l^{a}_{1}(\alpha)\wedge[\beta_{1},\ldots,\beta_{m}].
\end{equation}

The second term on the l.h.s.  
equals 
\begin{equation}\label{eqtwo}
\frac{s_{m-1,|\alpha|}^a}{2}\alpha \wedge d[\beta_1,\dots,\beta_m].
\end{equation}
 To see this,
 we use  the computation 
\begin{align}\label{eq:computationChris}
&\sum_{1\le p<q\le m}(-1)^{p+q}[l_2^b(\beta_p,\beta_q),\beta_1,\dots,\widehat{\beta_p}, \dots,\widehat{\beta_q},
\dots,{\beta_m}]\\
=&
\vs(m-1) \sum_{1\le p<q\le m} (-1)^{p+q}\iota(X_{l_2^b(\beta_p,\beta_q)}\wedge X_{\beta_1}\wedge\dots\wedge\widehat{X_{\beta_p}}\wedge \dots\wedge \widehat{X_{\beta_q}}\wedge \dots\wedge{X_{\beta_m}})\omega_b \nonumber\\
=& \vs(m-1)(-1)^{m} d\iota(X_{\beta_1}\wedge\dots\wedge X_{\beta_m})\omega_b \nonumber\\
=& d[\beta_1,\dots,\beta_m].\nonumber
\end{align}
where we used \cite[Lemma 9.2]{FRZ} in the second equality and  $\vs(m-1)\vs(m)=(-1)^m$.

The third term on the l.h.s. reads 
\begin{equation}\label{eqtwobis}
-  \frac{s^b_{1,2-m}}{2}[\alpha]\wedge[\beta_1,\dots,\beta_m]
-  \frac{s^a_{1,|\alpha|}}{2}\alpha\wedge[l_m^b(\beta_1,\dots,\beta_m)],
\end{equation}
where the second summand vanishes because of the assumption $m>2$.

The first term on the r.h.s. of eq. \eqref{eq:long}, using Lemma \ref{lem:square} and $-s^{A}_{0,0}=-s^{b}_{0,0}=1$, equals  \begin{equation}\label{eqthree}
-(-1)^{m(n_{a}+1)} 
[\alpha]\wedge[\beta_1,\dots,\beta_m].
\end{equation}

The last term on the r.h.s. is
\begin{equation}\label{eqfour}
\frac{s^a_{m,|\alpha|}}{2} l_1(\alpha\wedge[\beta_1,\dots,\beta_m])=
\frac{s^a_{m,|\alpha|}}{2}  \left(d\alpha\wedge[\beta_1,\dots,\beta_m]
+(-1)^{n_a-1-|\alpha|}\alpha\wedge d[\beta_1,\dots,\beta_m]\right).
\end{equation}

The term in  \eqref{eqtwo} cancels out with the second summand in 
  \eqref{eqfour}. Further,  using
that $l_1\alpha-[\alpha]=d\alpha$ by Remark  \ref{prop:iotaomega}
and the fact that $[\alpha]$ vanishes if $|\alpha|\neq 0$, one check that
the term \eqref{eqone} minus one half the term \eqref{eqthree} 
equals
$\frac{s^a_{m,|\alpha|}}{2}  d\alpha\wedge[\beta_1,\dots,\beta_m]$,  which is exactly the first summand in 
eq. \eqref{eqfour}. Finally, the term \eqref{eqtwobis} cancels out with one half the term \eqref{eqthree}.

Now, if $k=1,m=2$, then the term \eqref{eqtwo} is omitted (because the summand $i=2$ on the l.h.s. of
condition \eqref{eq:Linfinityconditions} is already given by the term \eqref{eqtwobis}),
and in \eqref{eqtwobis} the second summand no longer vanishes. We conclude that the case $\left\{ k = 1, m > 1\right\}$ 
 indeed works out with the choice of parameters given in \eqref{eq:sformula}.

One check in a similar way that the same holds for the case $\left\{ k > 1, m = 1\right\}$.
  This concludes the proof that $H$, as defined in the statement of the theorem,
  is
  an honest $L_{\infty}$-morphism.
\end{proof}


\section{A curved $L_{\infty}$-algebra associated to a closed differential form}\label{section:curved}


Let $(M,\omega)$ be a pre-$n$-plectic manifold. We saw  that one can associate to it
 an $L_{\infty}$-algebra $L_{\infty}(M,\omega)$, whose definition  we recalled in Def. \ref{def:multisymLinfinity}. For $k\ge 2$, the $k$-th multibracket of $L_{\infty}(M,\omega)$ is essentially given by contracting with $\omega$ the Hamiltonian vector fields of $k$ Hamiltonian forms, while the unary bracket is defined differently, as the de Rham differential. On the other hand, contracting with $\omega$ the Hamiltonian vector fields of an \emph{arbitrary number} of Hamiltonian forms is a natural operation, which we introduced in Def. \ref{def:square} using the notation $[\dots]$, and which proved to be necessary to describe the $L_{\infty}$-embedding obtained in Thm. \ref{thm:embed}. In this section we show that  the operation $[\dots]$ can be extended to a \emph{curved}  $L_{\infty}$-algebra structure canonically  associated to $(M,\omega)$, whose ``curvature'' is $-\omega$. 

\begin{defi} 
A {\bf curved $L_{\infty}$-algebra} is a $\ZZ$-graded vector space $W$ equipped with a collection ($k\ge0$) of linear maps $l_k \colon \otimes^kW\longrightarrow W$  of degree $2-k$, graded antisymmetric, and   satisfying for every  $m\ge 0$ and for every collection of homogeneous elements $w_1, \dots, w_m \in W$ the following relations: 
\begin{equation}\label{eq:jac}
\sum_{\substack{i+j=m+1\\i\ge 0,j\ge 1}}(-1)^{i(j-1)}\sum_{\sigma\in S_{i,m-i}} (-1)^{\sigma}\epsilon(\sigma)
l_j(l_i(w_{\sigma(1)}, \dots,w_{\sigma(i)}),w_{\sigma(i+1)}, \dots, w_{\sigma(m)} )=0.
\end{equation}
\end{defi}
 
\begin{remark}
The zero-th bracket   $l_0\colon \RR\to W$ has degree $2$, and is determined by the element $l_0(1)\in W_2$,  which we refer to as the ``curvature''. Writing   $D:=l_1$, the relations \eqref{eq:jac} for $m=0$ and $m=1$ read as follows: $D(l_0(1))=0$, i.e. $l_0(1)$ is a $D$-closed element, and
$D^2(x)+l_2(l_0(1),x)=0$ for all $x\in W$, so   $D$ does not square to zero in general.
\end{remark}

\begin{prop}\label{prop:curved}
Let $\omega\in \Omega^{n+1}(M)$ be a pre-$n$-plectic form. 
Consider the graded vector space whose non-trivial components are
\[
C_{i} =
\begin{cases}
\langle\omega\rangle & \hbox{ for } i=2,\\
\Omega^n(M) & \hbox{ for } i=1,\\
\ham{n-1}{M} & \hbox{ for } i=0,\\
\Omega^{n-1+i}(M) & \hbox{ for }  1-n \leq i \leq -1, 
\end{cases}
\]
where $\langle\omega\rangle$ denotes the one-dimensional real vector space generated by $\omega$. We define multilinear maps $[\dots]_k\colon C^{\otimes k}\to\Omega^{n+1-k}(M)$ as follows:
\begin{itemize}
\item for $k\ge 0$ and $\alpha_1,\dots,\alpha_k$ of  degree zero:
 \begin{equation*}
[\alphadk{k}]_k=\vs(k) \iota(\vk{k}) \omega
\end{equation*}
\item for $k\ge 2$ and   $\alpha_1,\dots,\alpha_k$ of  degree zero: $$[\alpha_1,\dots,\alpha_i,-\omega,\alpha_{i+1},\dots,\alpha_k]_{k+1}=(-1)^i d[\alpha_1,\dots,\alpha_k]_k,$$
\end{itemize}
where $d$ denotes the de Rham differential.

Then $C$, together with the above collection of multibrackets
(all other multibrackets are declared to  vanish) 
is a curved $L_{\infty}$-algebra. We denote it
$L^{curv}_{\infty}(M,\omega).$

\end{prop}
  
\begin{remark}
In particular, the ``curvature'' is  $[1]_0=-\omega$ (where $1\in C^{\otimes 0}=\RR$).
 The differential $D:=[\;]_1$ gives rise to the following chain complex: 
$$ C^{\infty}(M)\overset{0}{\to}\dots  \overset{0}{\to}\ham{n-1}{M}\overset{-d}{\to}\Omega^n(M)\overset{0}{\to}\langle\omega\rangle.$$
 \end{remark}  
  
\begin{remark}\label{rem:02}
 Notice that for $k\ge 1$, the multibrackets $[\dots]_k$ vanish, except possibly in two cases: all entries of   have degree zero, or all entries of   have degree zero except for one, which has degree 2 (and hence is a multiple of $-\omega$).
\end{remark}
\begin{proof}
For all $k\ge 0$, the multibracket $[\dots]_k$ is graded skew-symmetric and of degree $2-k$. 

We need to check that the generalized Jacobi identities \eqref{eq:jac}  are satisfied, for all $m\ge 0$. For $m=0$ we have $D(-\omega)=0$ by degree reasons. So in the following we take  $m\ge 1$ and homogeneous elements $w_1, \dots, w_m \in C$. 
We argue similarly to Rogers' \cite[Proof of Thm. 5.2]{RogersL}.
By Rem. \ref{rem:02}, we can assume that all $w_i$ have degree zero or $2$.
Further, since $C_2=\langle\omega\rangle$ is one-dimensional and the l.h.s. of eq. \eqref{eq:jac} is graded skew-symmetric by construction, we may assume that at most one $w_i$ has degree $2$. Hence we just need to consider two cases.
\bigskip

\noindent{\bf Case 1: all $w_1, \dots, w_m$ have degree zero.}

Let $\alpha_1,\dots,\alpha_m\in C$ be elements of degree zero.

For $m=1$ we have $[-\omega, \alpha_1]+D(D(\alpha_1)) =0$, since both terms vanish.

For $m=2$ we have $[-\omega, \alpha_1,\alpha_2]
-([D\alpha_1,\alpha_2]-[D\alpha_2,\alpha_1])+D[\alpha_1,\alpha_2]=0$:
the first and last term cancel out, while the two middle terms vanish.

Now we assume that $m\ge 3$, and consider the various summands of the sum $\sum_{\substack{i+j=m+1}}$  on the l.h.s. of eq. \eqref{eq:jac}.

\begin{itemize}
\item For $j=1$ (so $i=m$): the  corresponding summand vanishes, since   $D[\alpha_1,\dots,\alpha_m]=0$, for $D$ vanishes in negative degrees.

\item For $j=2,\dots,m-2$ (so $i=m-1,\dots,3$): the corresponding summand  vanishes by degree reasons (see Rem. \ref{rem:02}), since
for $i\ge 3$ the  bracket $[\dots]_i$ takes elements of degree zero to   elements of negative degree.

\item For $j=m$ (so $i=1$):  the corresponding summand  vanishes by degree reasons  (see Rem. \ref{rem:02}), since $D$ has degree one.

\item For $j=m-1$ (so $i=2$): the corresponding summand     is 
\begin{equation}
\label{eq:chris15}
\sum_{\sigma\in S_{2,m-2}}(-1)^{\sigma}\epsilon(\sigma) [[\alpha_{\sigma(1)},\alpha_{\sigma(2)}],\alpha_{\sigma(3)}, \dots, \alpha_{\sigma(m)}]_{m-1}
=-d[\alpha_1,\dots,\alpha_m].
\end{equation}
The above equality is obtained from the computation \eqref{eq:computationChris}, recalling 
Remark \ref{prop:iotaomega}.

\item For $j=m+1$ (so $i=0$): the  corresponding summand is
$$[-\omega,\alpha_1,\dots,\alpha_m]=d[\alpha_1,\dots,\alpha_m]$$
and cancels out with the summand given by $j=m-1$.
\end{itemize}

\bigskip
\noindent{\bf Case 2: all $w_1, \dots, w_m$ have degree zero except for one, which has degree 2.} 

Fix $m \ge 1$, and let   $\alpha_1,\dots,\alpha_{m-1}\in C$ be elements of degree zero.
We consider the various summands of the sum $\sum_{\substack{i+j=m+1}}$  on the l.h.s. of eq. \eqref{eq:jac}, applied to  $\alpha_1,\dots,\alpha_{m-1},-\omega$.
For $i=0$, the corresponding summand vanishes, because $-\omega$   appears twice as an entry of the bracket $l_j$.

For all $i\ge 1$,
\begin{equation}\label{eq:omegaa}
[[-\omega,\alpha_1,\dots,\alpha_{i-1}]_i,\alpha_i,\dots,\alpha_{m-1}]_j=0.
\end{equation}
Indeed, by Remark \ref{rem:02}, we may assume that  
the inner bracket has degree zero or $2$, and in those cases it reads respectively
$[-\omega,\alpha_1,\alpha_2,\alpha_3]_4$ (which is an exact form, so   its Hamiltonian vector field vanishes) and 
$[-\omega,\alpha]_2 $ (which vanishes).
Further,  except in the case  $i=2$,
$$[[\alpha_1,\dots,\alpha_{i}]_i,\alpha_{i+1},\dots,\alpha_{m-1},-\omega]_j=0$$
 by degree reasons (again by Rem. \ref{rem:02}).

Hence we need to consider only   the summand on the l.h.s. of \eqref{eq:jac} corresponding to  $i=2$ (so $j=m-1$, and  necessarily  $m\ge 2$).
It is a sum over unshuffles $S_{2,m-2}$, however due to eq. \eqref{eq:omegaa} it reduces to
 \begin{align*}
&\sum_{\sigma\in S_{2,m-3}}(-1)^{\sigma}\epsilon(\sigma)[[\alpha_{\sigma(1)},\alpha_{\sigma(2)}],\alpha_{\sigma(3)}, \dots, \alpha_{\sigma(m-1)}, -\omega]_{m-1}\\
=&\sum_{\sigma\in S_{2,m-3}}(-1)^{\sigma}\epsilon(\sigma)(-1)^m d[[\alpha_{\sigma(1)},\alpha_{\sigma(2)}],\alpha_{\sigma(3)}, \dots, \alpha_{\sigma(m-1)}]_{m-2}\\
=&-(-1)^md\Big(d[\alpha_1,\dots,\alpha_{m-1}]\Big)=0.
\end{align*}
Notice that the second equality   is just eq. \eqref
{eq:chris15}.
%
%
\end{proof}

In conclusion, a pre-$n$-plectic form $\omega$ on $M$ gives rise to both 
the $L_{\infty}$-algebra $L_{\infty}(M,\omega)$ of Def. \ref{def:multisymLinfinity}
and the curved $L_{\infty}$-algebra
$L^{curv}_{\infty}(M,\omega)$ of Prop. \ref{prop:curved}. The underlying graded vector spaces are the same in degrees $\le 0$, but the one of $L^{curv}_{\infty}(M,\omega)$  also has components in degrees $1$ and $2$. Their higher brackets are
almost identical, but the underlying chain complexes are very different 
 and certainly not quasi-isomorphic. 

\begin{remark}
The relation between $L_{\infty}(M,\omega)$ and $L^{curv}_{\infty}(M,\omega)$ is not clear at this stage.
One can regard both  as curved $L_{\infty}$-algebras, and ask if there is a natural   morphism of curved $L_{\infty}$-algebras (see \cite[Def. 6]{KajSta}) between them. 
 
In the simplest case in which $\omega\in \Omega^2(M)$ is a symplectic form, we have that 
$L_{\infty}(M,\omega)=C^{\infty}(M)$ is a Lie algebra while $L^{curv}_{\infty}(M,\omega)=C^{\infty}(M)\oplus \ham{1}{M}\oplus\langle\omega\rangle $ is an $L_{\infty}$-algebra concentrated in degrees $0,1,2$.
There is no  morphism $g\colon L_{\infty}(M,\omega)\to L^{curv}_{\infty}(M,\omega)$: the first condition {such morphism} would have to fulfil is an equality of certain maps from $\RR$ to the degree two component of $L^{curv}_{\infty}(M,\omega)$, and this condition fails since $\omega\neq 0$.
In the opposite direction, 
there is   a strict morphism   $f\colon L^{curv}_{\infty}(M,\omega)\to L_{\infty}(M,\omega)$, which is zero except for the restriction of the unary component $f_1$ to degree zero elements, which reads $f_1|_{C^{\infty}(M)}=Id_{C^{\infty}(M)}$.  For an arbitrary pre-$n$-plectic form $\omega$, one can check that there exists no strict morphism   $f\colon L^{curv}_{\infty}(M,\omega)\to L_{\infty}(M,\omega)$ such that $f_1$ is the identity in degrees $\le 0$, and we do not know if there is {a} non-strict one with this property. 
\end{remark}
  
\bibliographystyle{habbrv} 
\bibliography{Product}
\bigskip
\end{document}